\documentclass[a4paper,english]{article}
\usepackage[T1]{fontenc}
\usepackage{textcomp}
\usepackage{lmodern}
\usepackage[utf8]{inputenc}
\usepackage{babel}
\usepackage{csquotes}
\usepackage{amsmath}
\usepackage{mathtools}
\usepackage{amssymb}
\usepackage{amsthm}
\usepackage{tikz}
\usetikzlibrary{graphs}
\tikzset{mynode/.style={shape= circle, fill = white, inner sep = 0pt, outer sep = 0pt, minimum size = 4pt,draw}}
\tikzset{DirEdgeStyle/.style={>=stealth,->,thick}}
\tikzset{every edge/.append style={thick}}
\tikzset{>=stealth}
\newtheorem{lemma}{Lemma}
\newtheorem{conjecture}{Conjecture}
\newtheorem{corollary}[lemma]{Corollary}
\newtheorem{proposition}[lemma]{Proposition}
\newtheorem{theorem}[lemma]{Theorem}
\theoremstyle{remark}
\newtheorem{definition}[lemma]{Definition}
\newtheorem{remark}[lemma]{Remark}
\DeclareMathOperator{\Aut}{Aut}
\DeclareMathOperator{\ball}{Ball}
\DeclareMathOperator{\cayl}{Cay}
\DeclareMathOperator{\centre}{Z}
\DeclareMathOperator{\ord}{order}
\DeclareMathOperator{\rank}{rank}
\DeclareMathOperator{\squareGroup}{Sq}
\DeclareMathOperator{\Stab}{Stab}
\DeclareMathOperator{\triangles}{N}
\DeclarePairedDelimiter{\abs}{\lvert}{\rvert}
\newcommand*{\bijection}{\mathcal{B}}
\newcommand*{\Center}[1]{\centre(#1)}
\newcommand*{\Ball}[2]{\ball_{#1}(#2)}
\newcommand*{\covers}{\twoheadrightarrow}
\newcommand*{\defegal}{\coloneqq}
\newcommand*{\defi}[1]{\emph{#1}}
\newcommand*{\DRR}{DRR}
\newcommand*{\gen}[1]{\langle#1\rangle}
\newcommand*{\GRR}{GRR}
\newcommand*{\iso}{\cong}
\newcommand*{\orCayley}[2]{\overrightarrow\cayl( #1,#2)}
\newcommand*{\Evenorderfct}{F}
\newcommand*{\Orietendorderfct}{\check F}
\newcommand*{\ORR}{ORR}
\newcommand*{\presentation}[2]{\langle#1\mid#2\rangle}
\newcommand*{\setst}[2]{\{#1 \mid #2\}}
\newcommand*{\SquareGroup}[1]{\squareGroup( #1)}
\newcommand*{\Tarski}{\mathcal T}
\newcommand*{\Triangles}[2]{\triangles_3(#1,#2)} 
\newcommand*{\unCayley}[2]{\cayl( #1,#2)}
\newcommand*{\colCayley}[2]{\cayl_{\mathrm{col}}( #1,#2)}
\newcommand*{\labCayley}[2]{\cayl_{\mathrm{lab}}( #1,#2)}
\newcommand*{\Z}{\mathbf{Z}}
\usepackage[colorlinks,breaklinks,bookmarks,plainpages=false,unicode=true]{hyperref} 
\title{Cayley graphs with few automorphisms}
\author{Paul-Henry Leemann, Mikael de la Salle}
\date{\today}
\hypersetup{pdftitle={Cayley graphs with few automorphisms}, pdfauthor={Paul-Henry Leemann, Mikael de la Salle}, pdfsubject={\GRR, \DRR{} and \ORR{} for infinite groups. Cayley graphs covering few other graphs},pdfkeywords={\GRR, \DRR, \ORR, Cayley graph, covering, automorphisms of graphs, generalized dihedral group, generalized dicyclic group, regular automorphism group}}
\begin{document}
\maketitle
%
%
%
%
%
%
%
%
%
%
\begin{abstract}
We show that every finitely generated group $G$ with an element of order at least $\bigl(5\rank(G)\bigr)^{12}$ admits a locally finite directed Cayley graph with automorphism group equal to $G$.
If moreover $G$ is not generalized dihedral, then the above Cayley directed graph does not have bigons.
On the other hand, if $G$ is neither generalized dicyclic nor abelian and has an element of order at least $(2\rank(G))^{36}$, then it admits an undirected Cayley graph with automorphism group equal to $G$.
This extends classical results for finite groups and free products of groups.
The above results are obtained as corollaries of a stronger form of rigidity which says that the rigidity of the graph can be observed in a ball of radius $1$ around a vertex.
This strong rigidity result also implies that the Cayley (di)graph covers very few (di)graphs. 
In particular, we obtain Cayley graphs of Tarski monsters which essentially do not cover other quasi-transitive graphs.

We also show that a finitely generated group admits a locally finite labelled unoriented Cayley graph with automorphism group equal to itself if and only if it is neither generalized dicyclic nor abelian with an element of order greater than $2$.
\end{abstract}
%
%
%
%
%
%
%
%
%
%
\section{Introduction}
We start with some terminology. Throughout this article, the term \defi{graph} will stand for simple undirected unlabelled graphs. So a graph is a pair $(V,E)$ where $V$ is a set (the vertex set) and $E \subset V \times V$ is a subset (the edge set) satisfying $(x,y) \in E$ if and only if $(y,x) \in E$, and $(x,x) \notin E$ for every $x,y \in V$. An automorphism of a graph $(V,E)$ is a permutation $\varphi$ of $V$ preserving the edge set: $(x,y) \in E$ if and only if  $(\varphi(x),\varphi(y))\in E$. A graph is said to be \defi{locally finite} if every vertex belongs to finitely many edges.

Given a group $G$, its \defi{rank} is the smallest cardinality of a generating subset.
The \defi{exponent} of $G$  is the smallest integer such that $g^n=1$ for every $g\in G$ if it exists and $\infty$ otherwise. A group $G$ is a \defi{generalized dicyclic group} if it is a non-abelian group, has an abelian normal subgroup $A$ of index $2$ and an element $x$ of order $4$ not in $A$ such that $xax^{-1}=a^{-1}$ for every $a\in A$.

We will mostly be interested in Cayley graphs of finitely generated (possibly infinite) groups, which are central objects in geometric group theory. However, some of our results hold for any groups.
The (right) \defi{Cayley graph} $\unCayley{G}{S}$ of a group $G$ endowed with a generating set $S$ not containing the identity is the graph with vertex set $G$ and set $E = \{(g,h) \mid g^{-1} h \in S\cup S^{-1}\}$. We insist that, for us, the fact that $S$ is a generating set and does not contain the identity is part of the definition of the Cayley graph. So in every statement involving Cayley graphs $\unCayley{G}{S}$, the set $S$ will be assumed to be generating and to not contain the identity, even if it is not explicitly mentioned.

The group $G$ naturally acts on $\unCayley{G}{S}$ by left multiplication.
This action is regular (free and transitive) on the set of vertices.

It is natural in this context to search for a generating set $S$ such that $G=\Aut(\unCayley{G}{S})$. Such a Cayley graph is called a \defi{graphical regular representation}, or \GRR. An easy verification (consider the inverse map) shows that abelian groups of exponent greater than $2$ cannot admit a \GRR.
As observed by Watkins, \cite{MR0280416}, generalized dicyclic groups also do not admit \GRR{}s. Watkins actually conjectured that, apart these two infinite families, these are only finitely many groups not admitting \GRR{}s, and that they are finite. Let us agree to call such a group \defi{exceptional}.

\begin{conjecture}[Watkins \cite{MR0422076}]\label{conj:watkins} There is an integer $n$ such that every group of cardinality at least $n$ and which is neither abelian of exponent greater than $2$ nor generalized dicyclic admits a \GRR{}.
\end{conjecture}
Although he did not state it, it is very tempting to expect that for finitely generated groups the \GRR{} in Conjecture \ref{conj:watkins} can be moreover assumed to be locally finite. Our main result proves a particular case of this conjecture, replacing the assumption on the cardinality of the group by a similar assumption on the supremum of the orders of its elements. So the only possible counterexamples among finitely generated groups are some Burnside groups, that is finitely generated groups with a uniform bound on the order of every element.
\begin{theorem}\label{thm:mainUndirected}
Let $G$ be a finitely generated group. Assume that $G$ is neither abelian, nor generalized dicyclic and moreover that $G$ has an element of order at least $(2\rank(G))^{36}$.
Then $G$ admits a locally finite Cayley graph whose only auto\-morphisms are the left-multiplications by elements of $G$.
\end{theorem}
More generally, we prove that every finite generating subset of $G$ can be ``slightly'' enlarged so as to obtain a \GRR{}, see Corollary \ref{Cor:GRR} for a detailed statement.

Several cases of Conjecture \ref{conj:watkins} have been known for a long time.

Most importantly, it has been known from the $1970$s for finite groups, thanks to combined efforts of, notably, Imrich, Watkins, Nowitz, Hetzel and Godsil, \cite{MR0255446,MR0295919,MR0280416,MR0319804,MR0363998,MR0344157,MR0392660,MR0457275,HetzelThese,MR642043}. Moreover the finite exceptional groups are completely understood: there are $13$, all of whom of cardinality at most $32$, see \cite{MR642043} and the references therein. 
These proofs use deeply the fact that the groups under consideration are finite, and they do not admit straightforward generalizations to infinite groups. For example, the proof of the existence of \GRR{}s uses the Feit-Thompson theorem that states that every finite group of odd order is solvable.
One of the main idea of \cite{MR642043} is that if $H$ is a ``nice'' subgroup of a solvable group $G$, then the existence of a \GRR{} for $H$ implies the existence of a \GRR{} for $G$.
Let us also mention that Babai and Godsil showed \cite{MR656006} that if $G$ is a nilpotent non-abelian group of odd order, asymptotically almost all Cayley graphs of $G$ are \GRR{}s. Amusingly, combining Theorem \ref{thm:mainUndirected} with Zelmanov's positive answer to the restricted Burnside problem, \cite{MR1159227,MR1119009}, we recover the following particular case: for every $d$, there are only finitely many exceptional finite groups of rank $d$.

On the other hand, Watkins showed \cite{MR0422076} that a free product of at least $2$ and at most countably many non-trivial groups has a \GRR.
Moreover, if the group in question is finitely generated, then the \GRR{} in question is locally finite.
Here the method used is to start with a free group and then consider quotients of it.

Finally, recent developments in the subject were made into at least two distinct directions. 
The first one was the study of the Cayley graphs of $G$ when $G$ is either a finite abelian group of exponent greater than $2$, \cite{MR3537032}, or a finite generalized dicyclic group, \cite{MR3266284}. In particular, it is proven that if $G$ is not isomorphic to $Q_8\times (\Z/2\Z)^n$, then for almost all Cayley graphs of $G$, the group $G$ has index $2$ in $\Aut(\unCayley{G}{S})$.
On the other hand, for $G\iso Q_8\times (\Z/2\Z)^n$, then for almost all Cayley graphs of $G$, the group $G$ has index $8$ in $\Aut(\unCayley{G}{S})$.
The other recent development is the study of graphical Frobenius representations of a group $G$.
They are graphs $(V,E)$ such that $G\iso\Aut(V,E)$, the group $\Aut(V,E)$ acts transitively on $V$, there exists some non-trivial elements in $\Aut(V,E)$ that fixes a vertex and no non-trivial elements of $\Aut(V,E)$ fixes more than $1$ vertex, see \cite{MR3864735}.

We hence partially recover existing results about existence of \GRR{}s for finite groups and free products and also extend it to more infinite groups. 
Moreover, we also show that if $G$ has elements of arbitrary large order, then every finite generating set $S$ is contained in a generating set $T$ such that $\unCayley{G}{T}$ is a \GRR{} for $G$, and we control the size and position of $T$ in terms of the size of $S$, see Corollary \ref{Cor:GRR} for a precise statement.
This result can be thought of as a weak form of the statement that asymptotically almost all Cayley graphs of the group are \GRR{}s.

The general strategy to prove Theorem \ref{thm:mainUndirected} is similar to the approach in most of the work cited so far. We recognize that a Cayley graph is a \GRR{} by looking at the subgraph induced by some neighbourhood of the identity and proving that, for a suitable generating set, this subgraph has no automorphism. The main novelty we introduce here is how we decompose this task into two independent steps, and how we proceed for the second step. In a first step that we call \emph{orientation-rigidity}, we show that if $G$ is not abelian of exponent greater than $2$ and not generalized dicyclic (without any hypothesis on the order of the elements or on the rank of $G$), then every generating set $S_0$ of $G$ can be enlarged to a set $S_1$ in a way that the unoriented labels of the edges in a small ball around the identity in $\unCayley{G}{S_1}$ allow to recover the oriented labels in a small ball around the identity in $\unCayley{G}{S_0}$. Here the \defi{unoriented label} (or colour) of an edge $(g,h)$ is the set $\{g^{-1}h,h^{-1}g\}$ and its \defi{oriented label} is $g^{-1}h$. We have also a converse, so we obtain a geometric characterization, via their Cayley graphs, of groups which are neither generalized dicyclic nor abelian with exponent greater than $2$, see Theorem \ref{Thm:Prerigidity}. This idea is somewhat related to the notion of colour-preserving automorphisms and CCA (Cayley colour Automorphism) graphs, see \cite{MR3604062,MR3751067} for more about CCA graphs.

In a second  step that we call \emph{colour-rigidity}, we count the number of triangles to which a given edge belongs. This is not new. But elaborating on a idea in \cite{delaSalleTessera}, we show that, by suitably enlarging a generating set \emph{in the direction of the powers of an element of large order} (see Lemma \ref{Lemma:9.2} and \ref{Lemma:9.3} for precise statements), this number allows to recover from the local graph structure of the new Cayley graph the unoriented label (colour) of an edge of the original Cayley graph. In this step, the only assumptions on the group are that it is finitely generated and that it has an element of large order, and the argument remains valid for abelian or generalized dicyclic groups.

By the nature of the proof just sketched, we are also able to deal with directed Cayley graph. For us, a \defi{directed graph} (or \defi{digraph}) will be always simple and unlabelled. So it is a pair $(V,E)$ with a set $V$ (the vertex set) and a subset $E \subset V \times V$ (the arc set) with empty intersection with the diagonal $\{(x,x) \mid x \in V\}$. An automorphism of a digraph $(V,E)$ is a permutation of $V$ which preserves $E$. The (right) \defi{directed Cayley graph} $\orCayley{G}{S}$ has vertex set $G$ and arc set $\{(g,gs)\mid g \in G, s \in S\}$. It is a \defi{digraphical regular representation}, or \DRR, if the only automorphism of $\orCayley{G}{S}$ are the left multiplications by elements of $G$. Observe that there is a distinction between $\orCayley{G}{S}$ and $\unCayley{G}{S}$ only when $S$ is not symmetric.

Since every automorphism of $\orCayley{G}{S}$ naturally extends to $\unCayley{G}{S}$, the existence of a \GRR{} implies the existence of a \DRR, but the converse does not hold. The analogous of Conjecture \ref{conj:watkins} for \DRR{}s is known thanks to the work of Babai: there are only $5$ groups not admitting a \DRR{}: $(\Z/2\Z)^2$, $(\Z/2\Z)^3$, $(\Z/2\Z)^4$, $(\Z/3\Z)^2$ and the quaternion group $Q_8$. See \cite{MR603394} for finite groups and \cite{MR0498225} for infinite groups. However, even for finitely generated infinite groups, the \DRR{}s constructed in \cite{MR0498225} are never locally finite. Let us also mention the following recent result by Morris and Spiga in \cite{2018arXiv181107709M}: for a finite group $G$ of cardinality $n$, the proportion of subsets $S$ of $G$ such that $\orCayley{G}{S}$ is a \DRR{} goes to $1$ as $n\to\infty$.

A variation of the notion of directed graphs is the notion of oriented graphs: directed graphs without bigons, where a bigon in a digraph $(V,E)$ is an arc $(x,y)$ such that $(y,x)$ is also an arc.
For a Cayley digraph  $\orCayley{G}{S}$, this is equivalent to the fact that $S\cap S^{-1}$ is empty.
An oriented graph $\orCayley{G}{S}$ which is a \DRR{} is called an \defi{oriented regular representation} or \ORR.
Generalized dihedral groups do not admit a generating set without elements of order $2$ and thus cannot have \ORR.
For finite groups, Morris and Spiga showed, \cite{MR3873496}, that these are, alongside $11$ groups of order at most $64$, the only groups that do not admit an \ORR, answering a question posed by Babai in 1980.
Their proof relies on the classification of finite simple groups.

\begin{theorem}\label{thm:mainDirected}
Let $G$ be a finitely generated group. 
If $G$ contains an element of order at least $(5 \rank(G))^{12}$, then it has a locally finite directed Cayley graph whose only automorphisms are the left-multiplication by elements of $G$.

If $G$ is not a generalized dihedral group, then this digraph can be chosen without bigons.
\end{theorem}

Another way to study Cayley graphs is via the graphs they cover. Recall that if $\Gamma = (V,E)$ and $\Delta=(W,F)$ are graphs, a \defi{graph covering} $\varphi \colon \Gamma \covers \Delta$ is a map $V \to W$ such that, for every $v \in V$, the restriction of $\varphi$ to the set $\{v' \in V \mid (v,v') \in E\}$ of neighbours of $v$ is a bijection onto the set of neighbours of $\varphi(v)$. We warn the reader that this definition is well adapted to our setting of simple graphs, but it has to be adapted if one wants to include also non simple graphs, see for example \cite{MR3463202} for details. One way to construct covers of a Cayley graph $\Gamma = \unCayley{G}{S}$ is by taking the quotient by a subgroup $H \leq G$ (hence obtaining the so-called (right) Schreier coset graphs, which is a simple graph only under some conditions on $H$), but in general $\Gamma$ may cover other graphs.
In particular, in \cite{delaSalleTessera}, the second author together with Romain Tessera proved that if $G$ is finitely presented and contains an element of infinite order, then there exist a finite generating set $S$ and an integer $R$ (depending only on the rank of $G$) such that every graph with the same balls of radius $R$ as $\unCayley{G}{S}$ is covered by it.
On the other hand, the first author showed in \cite{MR3463202} that Cayley graphs of Tarski monsters do not cover other infinite transitive (even non simple) Schreier coset graphs.

The method of proof of Theorem \ref{thm:mainUndirected} also allows to construct generating sets such that essentially every covering $\unCayley{G}{S} \to \Delta$ comes from a Schreier graph. This allows us to show that for Tarski monsters there is essentially no covering $\unCayley{G}{T}\covers \Delta$ with $\Delta$ transitive, see Theorem \ref{Thm:Tarski} for a precise statement.

Finally, we stress the fact that given an oracle for the word problem in $G$, all our proofs are constructive and when we construct $T$ starting from $S$, there is an explicit bound on the cardinality of $T^\pm$, which depends only on the cardinality of $S^\pm$.

The next section contains all the necessary definitions and the statements of the main technical results of this paper. 
Section \ref{Section:Prerigidity} deals with orientation-rigidity.
Then, Section \ref{Section:MorePrerigidity} provides better bounds than the ones given by Theorem \ref{Thm:Prerigidity} for a large subclass of groups.
Finally, Section \ref{Section:StrongRigidTriples} deals with colour-rigidity.

\paragraph{Acknowledgements:} We are grateful for the referees for their very careful reading, their numerous and useful comments that helped us to greatly improve the exposition as well as some statements, in particular Theorem \ref{Thm:Prerigidity}.

The first author was partly supported by Swiss NSF grant P2GEP2\_168302.
The second author's research was supported by the ANR projects GAMME (ANR-14-CE25-0004) and AGIRA (ANR-16-CE40-0022).
Part of this work was performed within the framework of the LABEX MILYON (ANR-10-LABX-0070) of Universit\'e de Lyon, within the program "Investissements d'Avenir" (ANR-11-IDEX-0007) operated by the French National Research Agency (ANR).
\section{Definitions and results}\label{Section:Definition}
\subsection{Graphs}
We refer to the introduction for basic definitions of graphs. We now introduce some more terminology that we will use.

If $v$ is a vertex in a graph $(V,E)$, we denote by $\Ball{(V,E)}{v}$ the ball of radius $1$ around $v$, defined as the induced subgraph on the subset $V'=\{v\} \cup \{w \in V \mid (v,w) \in E\}$. Recall that induced subgraph means $\bigl(V',E \cap (V' \times V')\bigr)$.
\subsection{Groups and Cayley graphs}
Let $G$ be a group and $S \subset G$ a generating set. For simplicity and clarity of the exposition, we will suppose that our generating sets never contain the identity element. We will denote $1_G$ the identity in $G$, or simply $1$ when no confusion is possible. We will denote by $S^\pm=S\cup S^{-1}$ the symmetrisation of $S$, where $S^{-1}=\setst{s^{-1}}{s\in S}$. We will also denote by $S^{\leq n}$ the set of non-trivial elements of $G$ which can be written as a product of at most $n$ elements of $S^\pm$, that is vertices distinct from $1_G$ in the ball centred at $1$ and of radius $n$ in $\unCayley{G}{S}$.
In particular, the set $S^{\leq n}$ is always symmetric.

If follows from the definitions that $\Aut(\unCayley{G}{S})$, the automorphism group  of the Cayley graph $\unCayley{G}{S}$, is the group of permutations $\varphi$ of $G$ satisfying $\varphi(g S^\pm) = \varphi(g)S^\pm$ for every $g \in G$. We distinguish three natural subgroups of it.
\begin{itemize}
\item The subgroup $\Aut(\orCayley{G}{S})$. It coincides with the permutations $\varphi$ satisfying $\varphi(gS) = \varphi(g) S$ for every $g \in G$.
  \item The group $\Aut(\colCayley{G}{S})$ of \defi{colour-preserving automorphisms}, defined as the permutations $\varphi$ satisfying $\varphi(g s) \in \{\varphi(g)s,\varphi(g) s^{-1}\}$ for every $g \in G$ and $s \in S^\pm$.
\item The group $\Aut(\labCayley{G}{S})$ of \defi{labelled-preserving automorphisms}, defined as the permutations of $G$ satisfying $\varphi(g s) = \varphi(g) s$ for every $g \in G$, $s \in S^\pm$. This group is isomorphic to $G$ acting by left-translation because $S$ is generating.
\end{itemize}
As the notation suggests, the colour-preserving and label-preserving automorphisms correspond to automorphisms of some structure (coloured and labelled graphs respectively), but we do not elaborate on that as we will only work with groups that are concretely defined as groups of permutations.

More generally, if $B$ is a subgraph of $\unCayley{G}{S}$, we call
\begin{itemize}
\item \defi{$S$-orientation-preserving automorphism} of $B$ any permutation of the vertex set of $B$ satisfying $\varphi(gs) \in \varphi(g)S$ for every $g \in G$ and $s \in S$ such that $(g,gs)$ is an edge in $B$.
\item  \defi{colour-preserving automorphism} of $B$ any permutation of the vertex set of $B$ satisfying $\varphi(gs)  \in \{\varphi(g)s,\varphi(g)s^{-1}\}$ for every $g \in G$ and $s \in S^\pm$ such that $(g,gs)$ is an edge in $B$.
  \end{itemize}

We have the following obvious inclusions.

\includegraphics{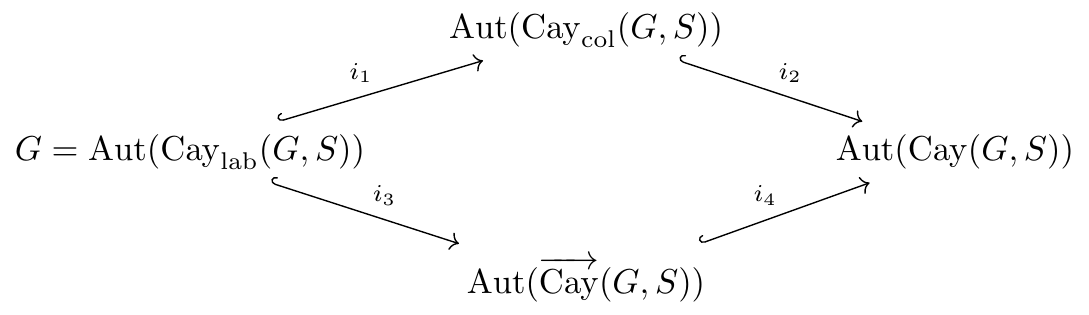}

In general, these four groups are distinct.
This can be seen by looking at $G=F_n$ the free group of rank $n$ and $S$ a free generating set. In that case $\unCayley{G}{S}$ is an infinite tree of valency $2n$.
Then the vertex stabilizer of $1$ is trivial for $G = \Aut(\labCayley{G}{S})$.
On the other hand, for any vertex $v$ at distance $k$ from the origin, the cardinality of its orbit under the action of the stabilizer of $1_G$ in $\Aut(\orCayley{G}{S})$ (respectively $\Aut(\unCayley{G}{S})$ and $\Aut(\colCayley{G}{S})$ is $n^k$ (respectively $(2n)^k$ and $2^k$).

The next Definition introduces the corresponding vocabulary, as well as some local and relative versions that will be important.
\begin{definition}\label{definition:Prerigidityetc}
A pair $(G,S)$ of a group with a generating set is an \defi{orientation-rigid} pair if $\Aut(\colCayley{G}{S})$ acts freely transitively on the vertices of the graph $\unCayley{G}{S}$, that is if $i_1$ is an equality. The group $G$ is \defi{orientation-rigid} if $(G,G)$ is an orientation-rigid pair.

The pair $(G,S)$ is said to be a \defi{colour-rigid} pair if $\Aut(\unCayley{G}{S})$ acts by colour-preserving automorphisms, that is if the inclusion $i_2$ is an equality.

A triple $(G,S,T)$ with two generating sets $S\subseteq T$ is said to be a \defi{strong orientation-rigid} triple if every colour-preserving automorphism of the ball $\Ball{\unCayley{G}{T}}{1_G}$ fixing the identity fixes the ball $\Ball{\unCayley{G}{S}}{1_G}$ pointwise.

The triple $(G,S,T)$ is said to be a \defi{strong colour-rigid triple} if every automorphism of the ball $\Ball{\unCayley{G}{T}}{1_G}$ fixing the identity acts by colour-preserving automorphisms on the ball $\Ball{\unCayley{G}{S}}{1_G}$.

The triple $(G,S,T)$ is said to be a \defi{strong \GRR{} triple} if every automorphism of the ball $\Ball{\unCayley{G}{T}}{1_G}$ fixing the identity fixes the ball $\Ball{\unCayley{G}{S}}{1_G}$ pointwise.

\end{definition}

We have the following elementary properties which justify the terminology.
\begin{lemma}\label{lemma:strongPtriplePpair} Let $S_0 \subset S \subset T$ be three generating sets of a group $G$.

  If $(G,S,T)$ is a strong orientation-rigid triple, then $(G,T)$ is an orientation-rigid pair.

 If $(G,S,T)$ is a strong colour-rigid triple, then $(G,T)$ is a colour-rigid pair.

 If $(G,S,T)$ is a strong \GRR{} triple, then $\unCayley{G}{T}$ is a \GRR{}.

 If $(G,S,T)$ is strong colour-rigid and $(G,S_0,S)$ is strong orientation-rigid, then $(G,S_0,T)$ is a strong \GRR{} triple.
\end{lemma}
\begin{proof} Let us only prove the first statement; the proofs of the second and third statements are similar, and the last is obvious. Assume that $(G,S,T)$ is a strong orientation-rigid triple. We wish to prove that $\Aut(\colCayley{G}{T})$ acts freely on $G$. Let $\varphi$ be in the stabilizer of $1_G$. If $\varphi$ was not the identity, there would exist $g \in G$ and $s \in S$ such that $\varphi(g) = g$ but $\varphi(gs) \neq gs$. Conjugating by the left-translation by $g$, we can assume that $g=1$, which gives by restriction to $\Ball{\unCayley{G}{T}}{1_G}$ a contradiction with the assumption.
 
\end{proof}

On the other hand, being orientation-rigid is a strengthening of being CCA. Indeed, a pair $(G,S)$ is orientation-rigid if and only if it is CCA and there are no automorphisms of $G$ that preserve $S$ and the edge-colouring; i.e. map every $s$ in $S$ to $s$ or $s^{-1}$.

Let us list some other obvious properties.
\begin{proposition}\label{prop:reformulation_in_terms_of_CD} Let $S$ be a 
generating set of $G$.

  \begin{itemize}
  \item The Cayley graph $\unCayley{G}{S}$ is a \GRR{} if and only if all the inclusions $i_1$ to $i_4$ are equalities.
  \item The Cayley digraph $\orCayley{G}{S}$ is a \DRR{} if and only if the inclusion $i_2$ is an equality.
  \item If $S \cap S^{-1}$ consists only of elements of order $2$, then
    \[\Aut(\colCayley{G}{S}) \cap \Aut(\orCayley{G}{S}) = \Aut(\labCayley{G}{S}).\]
  \item If $(G,S)$ is a colour-rigid pair, then $S$ contains a generating subset $S'$ such that $\orCayley{G}{S'}$ is a \DRR{}.
    \end{itemize}
  \end{proposition}
  \begin{proof} The first three items are obvious from the definitions. The fourth follows from the second and third. Indeed, let $S' \subset S$ be minimal for the property that $S^\pm = (S')^\pm$. Then $\unCayley{G}{S} = \unCayley{G}{S'}$, so if the inclusion $i_2$ is an equality for $S$, then the same is true for $S'$. Moreover, $S' \cap S'^{-1}$ consists only of elements of order $2$, so by the third item we obtain
    \[ G = \Aut(\colCayley{G}{S'}) \cap \Aut(\orCayley{G}{S'}),\]
    which is $\Aut(\orCayley{G}{S'})$ because  $\Aut(\unCayley{G}{S'}) = \Aut(\colCayley{G}{S'})$.
         \end{proof}

As explained in the introduction, our strategy to prove Theorem \ref{thm:mainUndirected} will be to describe two different ways of enlarging a finite generating set, one allowing to recover the orientation of the edges from the knowledge of the colours (\emph{i.e.} ensuring that $i_1$ becomes an equality), and the other allowing to recover the colours from the graph structure (\emph{i.e.} ensuring that $i_3$ becomes an equality). By the previous proposition, the combination of the two procedures will imply Theorem \ref{thm:mainUndirected}, whereas the second procedure will imply Theorem \ref{thm:mainDirected}. The equality case in inclusion $i_3$ essentially follows from $i_2$ (Proposition \ref{prop:reformulation_in_terms_of_CD}), and the inclusion $i_4$ does not play a role in our work.

\subsection{Main technical results}
Let us recall (Watkins, \cite{MR0280416}) that $(G,S)$ is a \defi{Class II} pair if the group $\Aut(G,S)$ of automorphisms $\psi$ of $G$ such that $\psi(S)=S$ is non-trivial, while $G$ is a \defi{Class II group} if $(G,S)$ is a Class II pair for every symmetric generating $S$.
The group $\Aut(G,S)$ naturally injects into $\Stab_{\Aut(\unCayley{G}{S})}(1)$, so being Class II is an obstruction to admitting a \GRR{}. Watkins conjectured that this is the only obstruction and that finite groups split into groups admitting a \GRR{} and Class II groups. The conjecture was finally proved in 1978 by Godsil, constructing upon the works of many others, see the introduction for a list of references.

Our first main technical result is that orientation-rigid groups are the same as the union of abelian groups with an element of order greater than $2$ and of generalized dicyclic groups, and hence a proper subclass of Class II groups. This statement is already known for finite groups \cite{MR3027684} but seems new for infinite groups. Moreover we obtain some quantitative estimates that will be crucial for us.
\begin{theorem}\label{Thm:Prerigidity}
For a group $G$, the following are equivalent
\begin{enumerate}
\item\label{item:notgendicyabel} $G$ is neither generalized dicyclic, nor abelian with an element of order greater than $2$;
\item\label{item:prerigid} $G$ is orientation-rigid;
\item\label{item:prerigid_quantitative} for every (equivalently there exists a) generating set $S$, the pair $(G,S^{\leq 3})$ is orientation-rigid;
\item\label{item:rstrongprerigid} for every (equivalently there exists a) generating set $S$, the triple $(G,S,S^{\leq 3})$ is strongly orientation-rigid;
\end{enumerate}
\end{theorem}
The proof of Theorem \ref{Thm:Prerigidity} is found in Section \ref{Section:Prerigidity}.
\begin{remark}\label{remark:exponent3optimal} The statement of Theorem \ref{Thm:Prerigidity} is optimal in the sense that one cannot replace $S^{\leq 3}$ by $S^{\leq 2}$ in (\ref{item:prerigid_quantitative}) or in (\ref{item:rstrongprerigid}). See the end of Section \ref{Section:Prerigidity} for a counterexample.
\end{remark}
There is a subclass of the groups appearing in Theorem \ref{Thm:Prerigidity} for which it is possible to obtain a far better bound; namely
\begin{proposition}\label{Prop:BetterBounds}
Let $G$ be a non-abelian finitely generated group.
Suppose that either $G$ has no elements of order $4$, or that it does not have non-trivial abelian characteristic subgroups.
Then for every finite symmetric generating set $S$, there exist two symmetric generating sets $\tilde S \subset T$, such that $\abs{\tilde S}=\abs S$, $\abs T\leq3\abs S$ and $(G,\tilde S,T)$ is strongly orientation-rigid.
\end{proposition}
The proof of Proposition \ref{Prop:BetterBounds} is found in Section \ref{Section:MorePrerigidity}.

Our second main technical result deals with colour-rigidity.
Before stating it we need a little bit of notation.
We note $\Evenorderfct(n)\defegal 2(2n^2+3n-2)^2(2n^2+4n-1)$ and $\Orietendorderfct(n)\defegal 2(2n^2+3n-2)^2(2n^2+4n)=4(2n^2+3n-2)^2(n+2)n$.

\begin{theorem}\label{Thm:Postrigidity}Let $G$ be a finitely generated group and $S$ a finite generating set.
If $G$ contains an element of order at least $\Evenorderfct(2\abs{S}^2+28\abs{S}-4)$, then there exists a generating set $T$ containing $S$, with at most $2\abs{S}^2+28\abs{S}$ elements and such that $(G,S,T)$ is a strong colour-rigid triple.

If $S$ contains no elements of order $2$ and $G$ has an element of order at least $\Orietendorderfct(2\abs{S}^2+28\abs{S}-4)$, then $T$ can moreover be taken with no element of order~$2$.
\end{theorem}
The proof of Theorem \ref{Thm:Postrigidity} is found in Section \ref{Section:StrongRigidTriples}.

\subsection{Proof of Theorem \ref{thm:mainUndirected} and Theorem \ref{thm:mainDirected}}
Before we prove Theorem \ref{Thm:Prerigidity} and Theorem \ref{Thm:Postrigidity}, we state and prove two corollaries, which imply the results stated in the Introduction. 

Recall that a \defi{generalized dihedral group} is the semi-direct product $A\rtimes \Z/2\Z$ where $A$ is abelian and $\Z/2\Z$ acts on $A$ by inversion.
\begin{corollary}\label{Cor:DRR}
  Let $G$ be a finitely generated group.
If $G$ contains an element of order at least $\Evenorderfct(2\rank(G)^2+28\rank(G)-4)$ then it has a \DRR{} of valency at most $2\rank(G)^2+28\rank(G)$.
If $G$ is not a generalized dihedral group and contains an element of order at least $\Orietendorderfct(2\rank(G)^2+28\rank(G)-4)$ then the above \DRR{} is in fact an \ORR.

If $G$ contains elements of arbitrary large order, then for every finite generating set $S$, there exists $T$ such that $T^\pm$ contains $S^\pm$, has at most $2\abs{S}^2+28\abs{S}$
 elements and $\orCayley{G}{T}$ is a \DRR{} for $G$.
If moreover $G$ is not a generalized dihedral group and $S$ has no elements of order $2$, then $\orCayley{G}{T}$ is an \ORR{} for $G$.
\end{corollary}
\begin{proof} We justify the first statement. The other assertions are proved similarly. For the \ORR{}, we use in addition that, as proved in \cite[Lemma 2.6]{MR3667669} by Morris and Spiga, any finitely generated group that is not generalized dihedral admits a generating set $S$ without elements of order $2$ and with $\abs S=\rank(G)$.

Let $S$ be a generating set of $G$ of cardinality $\rank(G)$. By Theorem \ref{Thm:Postrigidity}, there is $T$ containing $S$ of cardinality at most $2\rank(G)^2+28\rank(G)$ such that $(G,S,T)$ is a strong colour-rigid triple. By Lemma \ref{lemma:strongPtriplePpair}, $(G,T)$ is a colour-rigid pair. The conclusion follows by Proposition \ref{prop:reformulation_in_terms_of_CD}.
\end{proof}
Since cyclic groups always admit an \ORR, it is possible in Corollary \ref{Cor:DRR} to suppose that $\rank(G)\geq 2$ and a direct computation gives $\Orietendorderfct(2\rank(G)^2+28\rank(G)-4)\leq \bigl(5\rank(G)\bigr)^{12}$ as stated in Theorem \ref{thm:mainDirected}. 
\begin{corollary}\label{Cor:GRR}
Let $G$ be a finitely generated group which if not generalized dicyclic nor abelian.
If $G$ contains an element of order at least $\Evenorderfct(32\rank(G)^{6})$, then it has a \GRR{} of valency at most $32\rank(G)^{6}$.
If $G$ has no elements of order $4$ or no non-trivial characteristic abelian subgroups, then this bound can be lowered to $\Evenorderfct(18\rank(G)^2+84\rank(G)-4)$ and the \GRR{} obtained is of valency at most $18\rank(G)^2+84\rank(G)$.

If moreover $G$ has elements of arbitrary large order, then for every finite generating set $S$, there exists a symmetric set $T$ containing $S$ such that $\unCayley{G}{T}$ is a \GRR{} for $G$ with $\abs T \leq 32\abs{S}^6$, or $\abs T\leq 18 \abs{S}^2+84\abs{S}$ if $G$ has no elements of order $4$ or no non-trivial characteristic abelian subgroups. 
\end{corollary}
\begin{proof} Let $S_0\subset G$ be a finite generating set of cardinality $\rank(G)$. By Theorem \ref{Thm:Prerigidity}, $(G,S_0,S_0^{\leq 3})$ is a strong orientation-rigid triple. So is $(G,S_0,S)$ if $S \subset S_0^{\leq 3}$ is any subset such that $S^\pm = S_0^{\leq 3}$. One may pick such an $S$ of cardinality $4 \rank(G)^3 - 2\rank(G)^2+\rank(G)$ (worst case, when $G$ is a free group). An elementary computation yields $2|S|^2+28|S| \leq 32 \rank(G)^6$ (recall that, $G$ being not abelian, $\rank(G) \geq 2$). Theorem \ref{Thm:Postrigidity} implies the existence of $T$ containing $S$ with cardinality at most $32 \rank(G)^6$ such that $(G,S,T)$ is a colour-rigid triple. By Lemma \ref{lemma:strongPtriplePpair}, $(G,S_0,T)$ is a strong GRR triple. In particular $\unCayley{G}{T}$ is a \GRR.

  If $G$ has no elements of order $4$ or no non-trivial characteristic abelian subgroups, the same proof applies using Proposition \ref{Prop:BetterBounds} instead of Theorem \ref{Thm:Prerigidity}.

The second half of the Corollary is proved similarly.    \end{proof}
Since $G$ is not abelian, it is not cyclic and $\rank(G)\geq 2$.
A direct computation gives us $\Evenorderfct(32\rank(G)^{6}) \leq (2 \rank(G))^{36}$ which is the bound given in Theorem~\ref{thm:mainUndirected}.

\subsection{Coverings}
If $\Gamma$ is a graph, we say that a covering $\varphi\colon\unCayley{G}{T}\to \Gamma$ is compatible with the labels of $\unCayley{G}{S}$ if whenever two edges of $\unCayley{G}{S}$ have the same image under $\varphi$ they have the same labels. Such coverings are in bijection with conjugacy classes of subgroups of $G$ and that turns $\Gamma$ into a Schreier graph of $G$, see \cite{MR3463202}.
\begin{proposition}\label{Prop:Coverings}
Let $(G,S,T)$ be a strong \GRR{} triple. Then for any covering $\psi\colon\unCayley{G}{T}\to \Delta$ whose restrictions to balls of radius $1$ are isomorphisms, there exists a subgraph $\tilde\Delta$ of $\Delta$ such that the restriction of $\psi$ to $\unCayley{G}{S}$ is a covering onto~$\tilde\Delta$ which is compatible with the labels of $\unCayley{G}{S}$. Moreover $\tilde\Delta$ is globally invariant under the action of $\Aut(\Delta)$.
\end{proposition}
\begin{proof}
Let $\psi\colon\unCayley{G}{T}\to \Delta$ be a covering whose restrictions to balls of radius one are isomorphisms. Define $\tilde\Delta$ as the image of $\unCayley{G}{S}$. To prove the Proposition, we have to prove that, for every pair of edges $(g,h)$ and $(g',h')$ in $\unCayley{G}{T}$ which have the same image $(v,w)$ in $\Delta$, and such that $g^{-1} h \in S^\pm$, we have ${g'}^{-1} h'=g^{-1}h$. Denote by $\psi_g$ and $\psi_{g'}$ the restrictions of $\psi$ to $\Ball{\unCayley{G}{T}}{g}$ and $\Ball{\unCayley{G}{T}}{g'}$ respectively. By the assumption $\psi_g$ and $\psi_{g'}$ are isomorphisms onto $\Ball{\Delta}{v}$. So $\psi_{g'}^{-1} \circ \psi_g$ is an isomorphism between $\Ball{\unCayley{G}{T}}{g}$ and $\Ball{\unCayley{G}{T}}{g'}$ sending the center $g$ to $g'$ and $h$ to $h'$. Translating back to $\Ball{\unCayley{G}{T}}{1_G}$ and using the $(G,S,T)$ is a strong \GRR{} triple, we deduce that $\psi_{g'}^{-1} \circ \psi_g$ preserves all the $S$-labels. In particular $g^{-1}h=g'^{-1}h'$. The fact that $\tilde\Delta$ is globally invariant under the action of $\Aut(\Delta)$ is proved in the same way: every automorphism of $\Delta$ restricts in particular to an isomorphism of the balls, and in particular preserves the edges corresponding by $\psi$ to edges in $\unCayley{G}{S}$.
\end{proof}
We deduce the following Corollary.
\begin{corollary}\label{Cor:Coverings}
  Let $G$ be a finitely generated group which is neither generalized dicyclic nor abelian, and let $S$ be a finite generating set. 
Suppose that $G$ has an element of order at least $\Evenorderfct(32|S|^{6})$.
Then there exists a generating set $T$ of $G$ containing $S$ such that for any covering $\psi\colon\unCayley{G}{T}\to \Delta$ whose restrictions to balls of radius $1$ are isomorphisms, there exists a subgraph $\tilde\Delta$ of $\Delta$ such that the restriction of $\psi$ to $\unCayley{G}{S}$ is a covering onto~$\tilde\Delta$ which is compatible with the labels of $\unCayley{G}{S}$.
\end{corollary}
\begin{proof}
By the proof of Corollary \ref{Cor:GRR}, there exists $T$ containing $S$ such that $(G,S,T)$ is a strong \GRR{} triple. Proposition \ref{Prop:Coverings} applies.
\end{proof}

\subsection{Tarski monsters}
Let $p$ be a prime number. Recall that a Tarski monster of exponent $p$ is an infinite group such that every non-trivial proper subgroup is isomorphic to the cyclic group of order~$p$.
It follows easily from the definition that such a group is necessarily of rank $2$ and simple.
On the other hand, the existence of such groups is a difficult problem.
It was first solved by Ol'shanski\u\i{} in 1980, \cite{MR571100}, showing that for every $p>10^{75}$ there exist uncountably many non-isomorphic Tarski monsters of exponent $p$.
This bound was then lowered to $p\geq 1003$ by Adian and Lys\"enok in~\cite{MR1149884}.

It follows from Theorem \ref{thm:mainUndirected} that Tarski monsters admit \GRR{}s if $p>2^{72}$. By adapting the arguments in the specific case of Tarski monsters, we can obtain a much lower bound, which include all known Tarski monsters.
\begin{theorem}\label{Thm:Tarski2} Every Tarski monster of exponent greater than $263$ admits a \GRR{} of valency $24$.
\end{theorem}
Using Proposition \ref{Prop:Coverings} and Proposition 44 from \cite{MR3463202}, we obtain the following rigidity results about Cayley graphs of Tarski monsters. Both results are proved in Section \ref{Section:StrongRigidTriples}.
\begin{theorem}\label{Thm:Tarski}
For any prime $p>263$ and any Tarski monster $\Tarski_p$, there exists a generating set $T$ of size $12$ such that if $\psi\colon\unCayley{T_p}{T}\to \Delta$ is a covering whose restrictions to balls of radius $1$ are isomorphisms, then either $\psi$ is the identity, or $\Delta$ is infinite and the action of its automorphism group on its vertices has finite orbits.
In particular, if the covering is not trivial, then $\Delta$ is not transitive, and not even quasi-transitive.
\end{theorem}
Observe that our method, using triangles, is not able to say anything on coverings which are not isomorphisms in restriction to balls of radius $1$.

The motivation behind Theorem \ref{Thm:Tarski} comes from a question of Benjamini.
When studying the connective constant of a graph, he asked \cite{Benjamini} whether there exists a constant $m$ such that every infinite vertex transitive graph (not quasi-isometric to $\unCayley{\Z}{\{1\}}$) cover an infinite vertex transitive graph of girth (the size of the smallest cycle) at most $m$.
Even when dropping the assumption on the girth, the answer is not known.
Theorem \ref{Thm:Tarski} strongly suggests that some Cayley graphs of Tarski monster probably do not cover any other infinite transitive graphs.
However, this is not a definitive answer as Theorem \ref{Thm:Tarski} says nothing about covering whose restrictions to balls of radius $1$ are not isomorphisms.
Such coverings are exactly coverings that create new triangles.
If we consider the graphs under examination as topological spaces, they are the coverings with injectivity radius strictly less than $1.5$, but still greater than $1$ as the graphs under consideration are simple.
The classical example of such a ``bad'' covering between two simple graphs is given by $\unCayley{\Z}{\{1\}}\to \unCayley{\Z/3\Z}{\{1\}}$.
%
%
%
%
%
%
%
%
%
%
\section{Orientation-Rigid triples}
\label{Section:Prerigidity}
The aim of this section is to show Theorem \ref{Thm:Prerigidity}, which implies that generalized dicyclic groups and abelian groups of exponent greater than $2$ are the only groups which are not orientation-rigid.
We stress the fact that in this section, groups are not necessarily finitely generated.

Clearly, for a group $G$ and a symmetric generating set $T$, the stabilizer of $1_G$ in $\Aut(\colCayley{G}{T})$ coincides with the group 
of all permutations $\varphi$ of $G$ satisfying the following condition
\begin{equation*}\label{eq:CS_mik}
	\varphi(1)=1 \textnormal{ and } \forall g\in G,\forall s\in T, \varphi(gs) \in \varphi (g) \{s,s^{-1}\}.
\end{equation*}
Similarly, the group of colour-preserving automorphisms of the ball around $1_G$ in $\unCayley{G}{T}$ fixing $1_G$ is isomorphic to the group $\bijection(G,T)$ of permutations $\varphi\colon T \cup \{1_G\}\to T\cup \{1_G\}$ such that
\begin{equation*}
	\varphi(1_G)=1_G \textnormal{ and } \forall s \in T \cup \{1_G\},\forall t\in T, st\in T\implies\varphi(st)\in\varphi(s)\{t,t^{-1}\}.
\end{equation*}
In particular, the pair $(G,S)$ is orientation-rigid if and only if $\bijection(G,S)$ is trivial.

It directly follows from the definition that if $S\subseteq T$ are two symmetric generating sets, then $\Stab_{\Aut(\colCayley{G}{T})}(1_G)\leq\Stab_{\Aut(\colCayley{G}{S})}(1_G)$.
In particular, if $(G,S)$ is orientation-rigid for some $S$, then $G$ is orientation-rigid. This proves the implication (\ref{item:prerigid_quantitative})$\implies$(\ref{item:prerigid}) in Theorem \ref{Thm:Prerigidity}. We have already observed after Definition \ref{definition:Prerigidityetc} that (\ref{item:rstrongprerigid})$\implies$(\ref{item:prerigid_quantitative}). The implication (\ref{item:prerigid})$\implies$(\ref{item:notgendicyabel}) is due to Watkins \cite{MR0280416,MR0422076}. Indeed, more is true: if $G$ is abelian of exponent greater than $2$ or generalized dicyclic, then there is an non identity group automorphism $\varphi$ of $G$ such that $\varphi(g) \in \{g,g^{-1}\}$ for every $g \in G$ (the inverse if $G$ is abelian, and the map equal to the identity on $A$ and the inverse of $G \setminus \{A\}$ if $G$ is generalized dicyclic).

So the difficult implication in Theorem \ref{Thm:Prerigidity} is (\ref{item:notgendicyabel})$\implies$(\ref{item:rstrongprerigid}). 

We start by stating a few lemmas that we will use. In many of them, the quaternion group
\[Q_8 = \presentation{i, j,k}{i^4=1,i^2=j^2=k^2=ijk}\]
plays a special role.
It is classical that the automorphism group of $Q_8$ acts transitively on pairs of generators.
This implies that if $\gen{g,h}$ is isomorphic to $Q_8$, then it is isomorphic to it by $i\mapsto g$ and $j\mapsto h$.
On the other hand, the following easy result allows us to easily detect if $\gen{g,h}$ is isomorphic to $Q_8$.
\begin{lemma}\label{Lemma:Q8}
Let $G=\gen{g,h}$. If $gh=hg^{-1}$ and $hg=gh^{-1}$, then $G$ is a quotient of $Q_8$.
If moreover $G$ contains an element of order greater than $2$ or is not abelian, then it is isomorphic to $Q_8$.
\end{lemma}
\begin{proof}
Recall that $Q_8$ is also given by the following presentation
\[
	Q_8=\presentation{i, j}{i^4=1,i^2=j^2,jij^{-1}=i^{-1}},
\]
and that all proper quotients of $Q_8$ are elementary abelian $2$-groups.
The equality $gh=hg^{-1}$ is equivalent to $hg^{-1}h^{-1}=g$ and thus to $hgh^{-1}=g^{-1}$. Hence, we only need to show that $g^4=1$ and $g^2=h^2$.
Now, $hg=gh^{-1}$ is equivalent both to $g=hgh$ and to $g=h^{-1}gh^{-1}$.
We have
\begin{align*}
	g^2&=hg^{-1}h^{-1}\cdot hgh=h^2\\
	g^2&=h^{-1}gh^{-1}\cdot hg^{-1}h^{-1}=h^{-2}
\end{align*}
which gives us both $g^2=h^2$ and $g^4=1$.
 \end{proof}
In the sequel, we fix two symmetric generating sets $S\subseteq T$ of $G$, and we fix $\varphi$ in $\bijection(G,T)$.

It is possible to partition $T \cup \{1_G\}$ into $A\sqcup B$, where
\begin{align*}
	A=A_\varphi=\setst{g\in T \cup \{1_G\}}{\varphi(g)=g}\\
	B=B_\varphi=\setst{g\in T}{\varphi(g)=g^{-1}\neq g}.
\end{align*}
By definition, every $h$ in $B$ satisfies $h^2\neq 1$.

\begin{lemma}\label{lemma:centralizerS_mik}
If $g\in A$ is such that $g^2\neq 1$, then $C_G(g)\cap T \cap gT \subseteq A$.
\end{lemma}
\begin{proof}
  Let $h$ be an element of $C_G(g)\cap T \cap gT$. So $\varphi(h) = \varphi( g (g^{-1} h))$ belongs both to $\{h,h^{-1}\}$ and $\{ h, g h^{-1} g=g^2 h^{-1}\}$. So if $h \in B$, then $h^{-1} = g^2 h^{-1}$, a contradiction as we assumed $g^2 \neq 1$.
 \end{proof}
\begin{lemma}\label{lemma:ASubgroupS_mik}
Let $g$ and $h$ be two elements of $A$ such that $gh \in T$.
If $\gen{g,h}$ is not isomorphic to $Q_8$, then $gh$ is in $A$.
\end{lemma}
\begin{proof}
We show that if $gh \in B$, then $\gen{g,h}=Q_8$.
Since $B$ is closed under inversion, we also have $h^{-1} g^{-1} \in B$.
In particular, $(gh)^2\neq 1$ and $(h^{-1} g^{-1})^2\neq 1$. So we have $gh = \varphi(h^{-1} g^{-1}) = h^{-1} g$, or equivalently $hg=gh^{-1}$, and $h^{-1} g^{-1} = \varphi(gh) = g h^{-1}$, or equivalently $gh=hg^{-1}$. 
Therefore, we can apply Lemma \ref{Lemma:Q8} using the fact that $gh$ is of order greater than $2$.
 \end{proof}
\begin{lemma}\label{lemma:anticommutantS}
Let $g\in A$ and $h$ in $B$ be such that $hg$ and $g^{-1}h$ are still in $T$.
Then we have $hgh^{-1}=g^{-1}$.
Moreover, if $hg$ is in $A$ we have $\gen{h,g}\iso Q_8$.
\end{lemma}
\begin{proof}
Assume first that $hg \in A$.
Then $hg$ and $g^{-1}$ both belong to $A$ but their product belongs to $B$, so by Lemma \ref{lemma:ASubgroupS_mik} $\gen{hg,g}$ (which is just $\gen{h,g}$) is isomorphic to $Q_8$.
In particular $hgh^{-1} =g^{-1}$.

Assume now that $hg \in B$. Then $\varphi(hg) = g^{-1} h^{-1} \in \{h^{-1}g,h^{-1}g^{-1}\}$.
If $g^{-1} h^{-1} = h^{-1}g$, we have $hgh^{-1}=g^{-1}$ as desired.
On the other hand, if $g^{-1}h^{-1}=h^{-1}g^{-1}$, then $g$ and $h$ commute.
By Lemma \ref{lemma:centralizerS_mik}, we have $g^2=1$ and therefore $hgh^{-1}=g=g^{-1}$.
 \end{proof}
\begin{lemma}\label{lemma:A_abelian} Let $g,h \in A$ and $f \in B$ be such that
\[
	\{gh,fg,g^{-1}f,fh,h^{-1}f,fgh,(gh)^{-1}f\} \subseteq T.
\]
Then $ghg^{-1} \in \{h,h^{-1}\}$.
More precisely, $\gen{g,h}$ is isomorphic to $Q_8$ if $gh$ belongs to $B$ while $g$ and $h$ commute if $gh$ belongs to $A$.
\end{lemma}
\begin{proof}
By Lemma \ref{lemma:ASubgroupS_mik}, either $\gen{g,h}$ is isomorphic to $Q_8$ or $gh$ belongs to $A$.
In the second case, by applying three times Lemma \ref{lemma:anticommutantS} we obtain
\[
	h^{-1} g^{-1} = f(gh)f^{-1} = (fgf^{-1}) (fhf^{-1})=g^{-1} h^{-1}
\]
and $g$ and $h$ commute.
 \end{proof}
\begin{proof}[of (\ref{item:notgendicyabel})$\implies$(\ref{item:rstrongprerigid}) in Theorem \ref{Thm:Prerigidity}]
Let $S$ be a generating set of $G$ and let $T=S^{\leq 3}$.
In order to use preceding results, we will need to carefully check that all the hypothesis of the form $gh\in T$ hold .

We assume that $(G,S,T)$ is not strongly orientation-rigid, and we wish to show that $G$ is either generalized dicyclic or an abelian group of exponent greater than $2$.
By the hypothesis, there exists $\varphi \in \bijection(G,T)$ and $s_0$ in $S$ such that $\varphi(s_0)\neq s_0$. This means that $s_0$ is in $B\cap S$. If $G$ is abelian, it is not of exponent $2$ as $s_0 \neq s_0^{-1}$ and the Theorem is proved. So we can assume that $G$ is not abelian.

We distinguish three cases.

\paragraph{Case 1} The subgroup generated by $A \cap S$ is abelian and there is $g_0 \in S^{\leq 2}$ such that $\varphi(g_0) \neq g_0^{-1}$. So $g_0 \in A$ and is of order $\geq 3$.

In that case, for every $g \in B \cap S^{\leq 2}$, and $h \in A \cap S^{\leq 1}$ (respectively  $g \in B \cap S^{\leq 1}$, and $h \in A \cap S^{\leq 2}$), Lemma \ref{lemma:anticommutantS} implies that $g h g^{-1} = h^{-1}$. In particular for every $s_1,s_2 \in B \cap S$ and $h \in A \cap S^{\leq 2}$,
\[ (s_1s_2) h (s_1s_2)^{-1} = s_1(s_2 h s_2^{-1}) s_1^{-1} = s_1 h^{-1} s_1^{-1} = h.\]
Taking $h =g_0$ we obtain that $(s_1s_2) g_0 (s_1s_2)^{-1} \neq g_0^{-1}$ and therefore $s_1s_2 \in A$. Taking $h=s_3 s_4$ for $s_3,s_4 \in B \cap S$ we deduce that the group generated by $\{ss' \mid s,s' \in B \cap S\}$ is commutative and centralizes $A \cap S$. So if we denote by $H$ the subgroup of $G$ generated by $(A \cap S) \cup \{ss'\mid s,s' \in B \cap S\}$, using the hypothesis that the subgroup generated by $A \cap S$ is abelian, we obtain that $H$ is abelian. Moreover $s_0 h s_0^{-1} = h^{-1}$ for every $h \in (A \cap S) \cup\{ss'\mid s,s' \in S\}$, so also for every $h \in H$. Finally, using that $s_0^2 \in H$ and that $G=\gen{B\cap S,A\cap S}=\gen{s_0,H}$, we obtain that $H$ has index $2$ in $G$ and that $G$ is generalized dicyclic.

The remaining cases are:
\paragraph{Case 2} $\varphi$ coincides with the inverse map on $S^{\leq 2}$.

\paragraph{Case 3} The subgroup generated by $A \cap S$ is not abelian.

The proof in Case 2 and Case 3 share a lot, namely the following.

\begin{lemma}\label{lem:specialQ8}
In both cases 2 and 3, there exist $i,j \in S$ which generate a group isomorphic to $Q_8$ and such that for every $s \in S$, the element $x_s\in G$ defined by
\[
x_s\defegal
\begin{cases}
s &\textnormal{if } [s,i]=[s,j]=1,\\
is & \textnormal{if } [s,i]=1,[s,j]\neq 1,\\
js & \textnormal{if } [s,i]\neq 1,[s,j]= 1,\\
ijs & \textnormal{if } [s,i]\neq1,[s,j]\neq 1
\end{cases}
\]
is of order $2$ and in the center of $G$.
\end{lemma}

Before we prove Lemma \ref{lem:specialQ8}, let us explain how it can be used to complete the proof of the theorem. We have already dealt with Case 1, so we can assume that we are in Case 2 or Case 3. So Lemma \ref{lem:specialQ8} applies, and we take its notation.
So we have proved that the group $\presentation{x_s}{s\in S}$ is an abelian group of exponent $2$ contained in the center of $G$. Its intersection with $\gen{i,j}$ is equal to $\{1,i^2\}$, and we can split it as $\{1,i^2\}\times H$ where $H$ is an abelian group of exponent $2$. Altogether, we have $G=\presentation{i,j,x_s}{s\in S}\iso Q_8\times H$. Such a group is generalized dicyclic as desired.

It remains to prove Lemma \ref{lem:specialQ8}.

\paragraph{Proof of Lemma \ref{lem:specialQ8} in Case 2} For every $s$ and $t$ in $S$, we have $st=\varphi(t^{-1}s^{-1})\in\{ts,ts^{-1}\}$ and similarly $ts\in\{st,st^{-1}\}$. By Lemma \ref{Lemma:Q8}, either $s$ and $t$ commute, or they generate a subgroup isomorphic $Q_8$. The same holds if $s \in S^{\leq 2}$ and $t \in S$. So every $s \in A\cap S^{\leq 2}$ is of order $2$ and belongs to the center of $G$.

Let $i$ and $j$ be two non-commuting elements of $S$ --- in particular, $\gen{i,j}\iso Q_8$ which implies that both $i$ and $j$ are in $B$. For every $s \in S$, define $x_s$ by the formula in Lemma \ref{lem:specialQ8}. Using that $sis^{-1} \in \{i,i^{-1}\}$ and similarly for $j$, a straightforward verification shows that $x_s$ always commute with both $i$ and $j$, and that $ij x_s \in S^{\leq 3}$ (as $ijx_s = ijs$, $ijis=js$, $ij^2s=i^{-1}s$ or $ijijs=i^2 s$ depending on the value of $x_s$). Similarly, $jx_s \in S^{\leq 3}$. So $\varphi(ijx_s)$ belongs to $\{ijx_s,(ijx_s)^{-1}=i^{-1}jx_s^{-1}\}$ while on the other hand $\varphi(ijx_s)=\varphi(i (jx_s))$ belongs to $\{i^{-1}jx_s,i^{-1}(jx_s)^{-1}=ijx_s^{-1}\}$. Of the four possible equalities, two imply that $i^2=1$ which is absurd, and the other two imply $x_s^2=1$.

It remains now to show that the $x_s$ are in the center of $G$. If $x_s\in\{s,is,it\}$ we have that $x_s \in A \cap S^{\leq 2}$ and therefore belongs to the center. If $x_s = ijs$ and $x_t \neq ijt$, then from what we just proved $x_t$ belongs to the center of $G$ and in particular commutes with $x_s$. This implies that $t$ commutes with $x_s$ as $x_s$ commutes with $i,j$.
It remains to show that $[x_s,t]=1$ when $x_s=ijs$ and $t=ijt$. Since neither $s$ nor $t$ commutes with $i$ we have $[st,i]=1$ which implies by Lemma \ref{lemma:anticommutantS} that $st$ is in $A$.
Therefore $st$ is of order $2$ and we have $(st)^2=1=i^4=s^2t^2$ which implies $1=[s,t]=[x_s,t]$.
This concludes the proof of Lemma \ref{lem:specialQ8} in Case 2.

\paragraph{Proof of Lemma \ref{lem:specialQ8} in Case 3} If the subgroup generated by $A \cap S$ is not abelian, there are $i,j \in A\cap S$ that do not commute.

For every $s$ in $B\cap S^{\leq 2}$ we have $sis^{-1}=i^{-1}$ by Lemma \ref{lemma:anticommutantS}.
Moreover, if $s$ is in $B\cap S$, then $\gen{s,i}\iso Q_8$.
Indeed, since $si$ commutes with $j$, Lemma \ref{lemma:centralizerS_mik} implies that $si \in B$, and Lemma \ref{lemma:ASubgroupS_mik} (with $g=si$ and $h=i^{-1}$) implies $\gen{s,i} \iso Q_8$. On the other hand, Lemma \ref{lemma:A_abelian} implies that for every $t$ in $A\cap S$, either $[t,i]=1$ or $\gen{t,i}\iso Q_8$. 
The same results hold for $j$.
In particular we have $\gen{i,j}\iso Q_8$, and $sis^{-1} \in \{i,i^{-1}\}$ and $sjs^{-1} \in \{j,j^{-1}\}$ for every $s \in S$.

We claim that $ij$ is in $B$.
Indeed, every element $s$ of $S$ commutes with at least one element $y$ of $\{i,j,ij\}$.
If $ij$ was in $A$, then by applying Lemma \ref{lemma:centralizerS_mik} to $s$ and $y$ we would have $s\in A$.
But this would imply that $\varphi$ is the identity when restricted to $S$, which is absurd.

As in Case 2, if for $s \in S$ we define $x_s$ by the formula in Lemma \ref{lem:specialQ8}, we obtain that $x_s$ always commutes with both $i$ and $j$, and we can check that $\{x_si,i^{-1}x_s,ijx_s, (x_s)^{-1}ij\}$ is contained in $S^{\leq 3}$.
Moreover, we have $x_s=ijs$ if and only if $s$ is in $B$.
Indeed, we already know that $s$ does not commute with $i$ or $j$ if $s$ is in $B$.
For the other direction, suppose that $s$ is in $A$ and $x_s=ijs$.
Then $ij$ commutes with $s$ and $s^2=i^2 \neq 1$. By Lemma \ref{lemma:centralizerS_mik} $ij$ belongs to $A$ which is absurd.

It remains to show that $x_s$ is of order $2$ and in the center of $G$.
We first show that $x_s$ belongs to $A$. If this was not the case, Lemma \ref{lemma:anticommutantS} for $g=i$ and $h=x_s$ would imply that $x_s i x_s^{-1} = i^{-1}$, contradicting the fact that $x_s i x_s^{-1} = i$. So applying Lemma \ref{lemma:anticommutantS} for $g=x_s$ and $h=ij$ we obtain $(ij) x_s (ij)^{-1} =x_s^{-1}$. But $x_s$ commutes with $ij$, so $x_s^{2} =1$.

Finally, we show that $x_s$ is in the center of $G$. 

If $x_s=s$, then for every $t$ in $S$, either $s$ and $t$ commute or they generate $Q_8$ which is impossible since $s$ is of order $2$ and thus $x_s$ is in the center of $G$.
If $x_s=is$ and $x_t=it$, then $(st)i(st)^{-1}=i$ which implies that $st$ is in $A$ and by Lemma \ref{lemma:A_abelian} that $1=[s,t]=[x_s,t]$.
If $x_s=is$ and $x_t=ijt$, then $t$ being in $B$ and $x_s \in A \cap S^{\leq 2}$ we obtain by Lemma \ref{lemma:anticommutantS} that $t x_s t^{-1} = x_s^{-1} = x_s$.

If $x_s=is$ and $x_t=jt$ we have $[x_s,t]=i^2[s,t]$ and therefore $[x_s,t]=1$ if and only if $s$ and $t$ do not commute.
Suppose that $s$ and $t$ do commute. Then $(st)^2 = s^2 t^2 = j^2 i^2 = 1$. This implies that $ist=i\cdot st$ is in $A$ as a product of an element of $A$ and of an element of order $2$.
Since $ij$ is in $B$ and $ij\cdot ist$ and $(ist)^{-1}\cdot ij$ are in $S^{\leq 3}$, Lemma \ref{lemma:anticommutantS} implies $(ij) (ist) (ij)^{-1} = (ist)^{-1}$, that is $jst= ijist = st i^{-1} ij=stj$ which is absurd.

At this point, we have proven that elements of the form $x_s=s$ and $x_s=is$ (or similarly $x_s=js$) are in the center of $G$. It remains to show that if both $s$ and $t$ are in $B$ then $x_s=ijs$ and $x_t=ijt$ commute.
However in this case $[x_s,x_t]=[s,t]$.
Since $st$ commute with $i$, it belongs to $A$.
By Lemma \ref{lemma:ASubgroupS_mik}, $ist$ is also in $A$.
Applying Lemma \ref{lemma:anticommutantS} to $h=ij$ et $g=ist$ (both $gh$ and $g^{-1}h$ are indeed in $S^{\leq 3}$) we have $(ij) (ist) (ij)^{-1} = (ist)^{-1}$.
Since $ij$ commutes with $st$, we can rewrite this last equality as $i^{-1} st = i^{-1} (st)^{-1}$.
That is, $st$ is of order $2$ and $(st)^2=1=i^4=s^2t^2$ which implies that $s$ and $t$ commute.
This concludes the proof of Lemma \ref{lem:specialQ8} in Case 3, and of Theorem \ref{Thm:Prerigidity}.
 \end{proof}
Finally, to justify Remark \ref{remark:exponent3optimal} we exhibit an infinite family of groups showing that it is not possible to replace $S^{\leq3}$ by $S^{\leq2}$ in Theorem \ref{Thm:Prerigidity}.
For every $n\geq 2$ let
\[
	H_n = \presentation{s_1,\dots,s_n}{\forall i\neq j: s_is_js_i^{-1}=s_i^{-1}}
\]
with generating set $S_n = \{s_1,\dots,s_n\}$ and let $\varepsilon=s_1^2$.
Such groups admit a normal form in the sense that every $g\in H_n$ admits a unique decomposition as $g=\varepsilon^{\alpha_0}\prod_{i=1}^ns_i^{\alpha_i}$ with all the $\alpha_i$ in $\{0,1\}$. 
This implies that $H_n$ has order $2^{n+1}$.
Moreover, if $m\leq n$, every choice of $m$ distinct elements in $\{s_1,\dots,s_n\}$ generates a subgroup of $H_n$ isomorphic to $H_m$.
Since $H_2=Q_8$ this implies that none of the $H_n$ are abelian.

While $H_2=Q_8$ and $H_3$ are generalized dicyclic (for $H_3$, the index $2$ abelian subgroup is $\gen{s_1s_2,s_3,\varepsilon}$), the group $H_{n}$ is never generalized dicyclic for $n\geq 4$.
Indeed, suppose that this is not the case and that $H_n$ is generalized dicyclic for some $n \geq 4$. In particular $H_n$ admits an index $2$ abelian normal subgroup $A$. Since $s_i$ and $s_j$ do not commute if $i \neq j$, they cannot both belong to $A$. So there are at least $n-1$ different $i$ such that $s_i$ does not belong to $A$. Without loss of generality we can assume that $s_1,\dots,s_{n-1}$ do not belong to $A$. Then ($n \geq 4$) both $s_1 s_2$ and $s_2s_3$ belong to $A$, and in particular commute. This is absurd as $[s_1s_2,s_2 s_3] = \varepsilon \neq 1$.

It remains to show that the function $\varphi\colon H_n \to H_n$ which sends $g$ to $g^{-1}$ is a colour-preserving automorphism of $\unCayley{H_n}{S_n^{\leq 2}}$. This is a simple verification, but for the reader's convenience we provide the proof. We have to show that $(gh)^{-1} \in \{g^{-1}h,g^{-1} h^{-1}\}$  for every $g \in H_n$ and $h \in S_n^{\leq 2}$, or equivalently $ghg^{-1} \in \{h,h^{-1}\}$. If $h = \varepsilon$ this is clear as $\varepsilon$ belongs to the center of $H_n$. Otherwise, $h$ is of the form $s_i^{\pm 1}$ or $s_i s_j$ for $i \neq j$, and in both cases $h^{-1} = \varepsilon h$. So to conclude it is enough to show that $ghg^{-1} \in \{h,\varepsilon h\}$ for every $g,h \in H_n$. If $g$ and $h$ are generators this is the definition of $H_n$, and the general case follows by writing $g$ and $h$ as products of the generators.
%
%
%
%
%
%
%
%
%
%
\section{More on orientation-rigidity}
\label{Section:MorePrerigidity}
In the last section, we have shown that if $G$ is neither a generalized dicyclic group nor an abelian group of exponent greater than $2$, then for every symmetric generating set $S$, the triple $(G,S,T)$ is strongly orientation-rigid for $T=S^{\leq 3}$.
The main caveat of this general method is that the size of $S^{\leq 3}$ grows fast.
This is important as in the hypothesis of Theorem \ref{Thm:Postrigidity} and Proposition \ref{Prop:Coverings} we suppose that $G$ has an element of order at least $\Evenorderfct(2\abs{T}^2+28\abs{T}-4)$ which is equivalent to $2^{10} |T|^{12}$ as $|T|\to \infty$. While the exponent $3$ is optimal in the general case, it is possible to obtain a far better result for many groups.
In this section, we provide criteria on $(G,S)$ which ensure that $(G,\tilde S,T)$ is strongly orientation-rigid for some $\tilde S\subset T$ with $\abs{\tilde S}=\abs S$ and $\abs T\leq 3\abs S$.
%
%
%
%
%
%
%
%
%
%
%
%
%
%
%
\begin{proposition}\label{Proposition:propertyR}
Let $G$ be a group and $S$ a finite generating set such that 
\[\forall s\in S, 
	\begin{cases}
	s^2=1 \textnormal{ or}\\
	\exists g\defegal g_s\in  G : s^2\neq g^2 \textnormal{ and } sgs^{-1}\notin\{g,g^{-1}\}.  
	\end{cases}
	\tag{*}\label{PropertyR}
\]
Let $p$ denote the number of elements of $S$ of order $2$ and $q$ the number of elements of $S$ of order at least $3$.
Then there exists $S\subseteq T$ with $\abs{T^\pm}\leq p+6q$ such that $(G,S,T)$ is strongly orientation-rigid.
\end{proposition}
\begin{proof}
Let $T$ be the union of  $S$, $\setst{g_s}{s\in S \textnormal{ of order at least } 3}$ and $\setst{h_s\defegal s^{-1}g_s}{s\in S \textnormal{ of order at least } 3}$.
Then $T^\pm$ contains at most $p+6q$ elements.

Suppose by contradiction that we have $\varphi$ a colour-preserving automorphism of the ball $\Ball{\unCayley{G}{T}}{1_G}$ fixing $1_G$ that does not pointwise fix elements of $S$.
Then we have some $s\in S$ with $\varphi(s)=s^{-1}\neq s$. In particular, $s$ is not of order $2$ and the action of $\varphi$ on the triangle $(1,s,g_s)$ is depicted in Figure \ref{Figure:TrianglePhi}.
\begin{figure}
\centering
\includegraphics{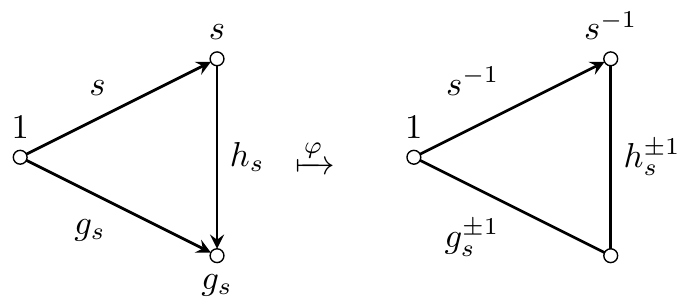}
\caption{The action of $\varphi$ on the triangle $(1,s,g_s)$.}
\label{Figure:TrianglePhi}
\end{figure}
Since we are looking at balls in a Cayley graph of $G$, the label of every cycle in it corresponds to a relation in $G$.
Therefore, depending on $\varphi$, one of the following relations is true in $G$:
\[
\begin{cases}
s^{-1}h_sg_s^{-1}=1\\
s^{-1}h_s^{-1}g_s=1\\
s^{-1}h_sg_s=1\\
s^{-1}h_s^{-1}g_s^{-1}=1
\end{cases}
\]
which correspond respectively to 
\[
\begin{cases}
s^2=1\\
sg_s=g_ss\\
s^2=g_s^2\\
sg_ss^{-1}=g_s^{-1}
\end{cases}
\]
which contradict our hypotheses.
 \end{proof}
In the rest of this section, we investigate some properties of the group $G$ or of the pair $(G,S)$ that imply Condition \eqref{PropertyR} on $S$.
In this context, elements of order $4$ as well as squares play an important role.
\begin{lemma}\label{Lemma:Order4}
Let $G$ be a group and $S$ a generating set such that
\[
	\begin{cases}
		\textnormal{all elements of $S\cap \Center{G}$ have order $2$ and}\\
		\textnormal{$S$ does not contain elements of order $4$.}
	\end{cases}
	\tag{\dag}\label{Condition:Order4}
\]
Then $S$ satisfies Condition \eqref{PropertyR}.
Moreover, if $hs\neq sh$, then it is possible to choose $g_s$ in $\{sh,sh^{-1},hs^{-1},h^{-1}s^{-1}\}$.
\end{lemma}
\begin{proof}
Let $s\in S$.
If $s^2=1$, then there is nothing to do.
If $s^2\neq 1$, then by assumption $s^4\neq 1$ and there exists $h\in G$ such that $hs\neq sh$, this implies that $h^{-1}s\neq sh^{-1}$.
On the other hand, for every $k\in G$, we cannot have $k^2= s^2$ and $(k^{-1})^2= s^2$ together, otherwise we would have $s^4=1$.
Suppose that $h^2\neq s^2$ (the proof is similar for $h^{-1}$) and take $g_s=sh^{-1}$.
By assumption on $h$, we have $sg_s=s\cdot sh^{-1}\neq sh^{-1}\cdot s=g_ss$.
Now, if $sg_ss^{-1}=g_s^{-1}$, we have $s\cdot sh^{-1}\cdot s^{-1}=(sh^{-1})^{-1}$ and  then $s^2=h^2$ which contradicts our hypothesis.
As noted before, at least one of the square of $g_s$ and $g_s^{-1}$ is not equal to $s^2$.
We thus have proved that if $sh\neq hs$ and $s^4\neq 1$, at least one of $\{sh^{-1},hs^{-1},sh,h^{-1}s^{-1}\}$ may be taken for $g_s$ in order to verify Condition \eqref{PropertyR}.
 \end{proof}
\begin{corollary}
If $G$ is such that 
\[
	G=\gen{G\setminus(\Center{G}\cup\setst{g\in G}{g^4=1})\cup\setst{g\in G}{g^2=1}}
\]
then it admits a generating set $S$ satisfying Condition \eqref{PropertyR}.
If moreover $G$ is finitely generated, then there exists a finite generating set $S$  satisfying Condition \eqref{PropertyR}.
\end{corollary}
As an important corollary of Lemma \ref{Lemma:Order4}, we have
\begin{proposition}\label{Prop:NoElements4}
Suppose that $G$ is not abelian and has no elements of order $4$.
Then for every generating set $S$, there exists $S'$ of the same cardinality as $S$ that satisfies Condition \eqref{PropertyR}.
\end{proposition}
\begin{proof}
Since $G$ is not abelian, every generating set $S$ contains some element $t$ outside the center and $S'\defegal\bigl(S\setminus \Center{G}\bigr)\cup\setst{st}{s\in S\cap \Center{G}}$ works.
 \end{proof}
Another way to look at Condition \eqref{PropertyR}  is to forget elements of order $4$ and turn our attention to squares of elements in $G$.
The proof of the following lemma is straightforward and left to the reader.
\begin{lemma}
Let $s,g$ be any two elements in $G$ such that $[s^2,g^2]\neq 1$.
Then, $s^2\neq 1$, $sg\neq gs$, $s^2\neq g^2$ and $sgs^{-1}\neq g^{-1}$.
\end{lemma}
Let $\SquareGroup{G}$ be the subgroup of $G$ generated by $\setst{g^2}{g\in G}$.
This is a fully characteristic subgroup of $G$ (invariant under all endomorphisms of $G$).
As a corollary of the last lemma, we have
\begin{corollary}\label{Corollary:Square}
Let $G$ be a group and $S$ a generating set.
If
\[
	S^2\cap \Center{\SquareGroup{G}}\subseteq \{1\}.\tag{\ddag}\label{Condition:Square}
\]
then $S$ satisfies Condition \eqref{PropertyR}.
\end{corollary}
If $S$ satisfies \eqref{Condition:Square} and $G$ has at most $1$ element of order $2$, then $S$ also satisfies \eqref{Condition:Order4}.
The proof is straightforward and left to the reader.

Since the center of a group is characteristic, we have that $\Center{\SquareGroup{G}}$ is an abelian and characteristic subgroup of $G$.
This implies the following:
\begin{proposition}\label{Prop:RigidGroup}
Let $G$ be a group without any non-trivial abelian characteristic subgroups (for example a non-abelian characteristically simple group).
Then every generating set $S$ satisfy Condition \eqref{PropertyR}.
\end{proposition}
The study of finite exceptional groups gives us information on the relative strength of our different rigidity criterion.
Simple verifications show that $Q_8\times \Z/4\Z$ has no generating set satisfying Condition \eqref{PropertyR} despite being orientation-rigid. 
On the other hand, $Q_8\times \Z/3\Z$ has no generating set satisfying \eqref{Condition:Square}, while $\{(i,1),(j,1)\}$ satisfies \eqref{Condition:Order4}.
%
%
%
%
%
%
%
%
%
%
\section{Colour-rigid triples}
\label{Section:StrongRigidTriples}
The aim of this section is to give a proof of Theorem \ref{Thm:Postrigidity}, as well as of Theorem \ref{Thm:Tarski2} and Theorem \ref{Thm:Tarski}.

We will make extensive use of the notion of triangles in graphs.
\begin{definition}
A \defi{triangle} in a graph $(V,E)$ is a complete subgraph with three vertices.
If $S \subset G$ is a finite generating set of a group and $s \in S$, we denote by $\Triangles{s}{S}$ the number of triangles in $\unCayley{G}{S}$ containing the vertices $1$ and $s$.
\end{definition}
Since $\Triangles{s}{S}$ is a geometric property, an automorphism of $\unCayley{G}{S}$ cannot send an edge labelled by $s$ to an edge labelled by $t$ if $\Triangles{s}{S}\neq\Triangles{t}{S}$.
Since $G=\Aut(\labCayley{G}{S})$, for every $s \in S$ the number of triangles containing $g$ and $gs$ does not depend on $g$ and is thus equal to $\Triangles{s}{S}$.
In particular (taking $g=s^{-1}$) we have $\Triangles{s}{S}=\Triangles{s^{-1}}{S}$.
We will sometimes say that $s$ belongs to $n$ triangles when we mean that $\Triangles{s}{S}=n$ or equivalently that every edge labelled by $s$ belongs to $n$ triangles.

We begin by proving two lemmas on triangles in Cayley graphs.
These lemmas are strongly inspired by Lemmas 9.2 and 9.3 of \cite{delaSalleTessera}.
The main difference is that we only require the group $G$ to have an element of large enough order, and not necessarily of infinite order.
Observe that a similar result was already announced in \cite{delaSalleTessera}, but without an actual proof and without an explicit lower bound on the order of the elements.

In contrast with the two preceding sections, all groups in this section are finitely generated and all generating sets under consideration are finite.
%
%
%
%
%
%
%
%
%
%
\subsection{The key lemmas}
\begin{lemma}\label{Lemma:9.2}
Let $S\subseteq G\setminus \{1\}$ be a finite symmetric generating set.
Suppose that $G$ has an element of order $o$, with
\begin{equation}\label{eq:assumption_on_order}\begin{cases}
    o \geq 2(2\abs S^2+3\abs S-2)^2(2\abs S^2+2\abs S) & \textnormal{if $o$ is odd}\\
    o \geq \Evenorderfct(\abs S)=2(2\abs S^2+3\abs S-2)^2(2\abs S^2+4\abs S-1)& \textnormal{otherwise}.
\end{cases}\end{equation}
Then, for each $s_0$ in $S$, there exists $S\subseteq S'\subseteq G$ a finite symmetric generating set such that
\begin{enumerate}\renewcommand{\theenumi}{\alph{enumi}}
\item $\Delta\defegal S'\setminus S$ has at most $4$ elements;\label{ConditionA}
\item $\Delta\cap\setst{s^2}{s\in S}=\emptyset$;\label{ConditionB}
\item $\Triangles{s}{S'}\leq 6$ for all $s\in\Delta$;\label{ConditionC}
\item $\Triangles{s}{S'}=\Triangles{s}{S}$ for all $s\in S\setminus\{s_0,s_0^{-1},s_0^2,s_0^{-2}\}$;\label{ConditionD}
\item\label{item:cases_Lemma92} the pair $\bigl(\Triangles{s_0}{S'}-\Triangles{s_0}{S},\Triangles{s_0^2}{S'}-\Triangles{s_0^2}{S}\bigr)$ belongs to
\[\begin{cases}
	\{(2,0),(4,0)\} & \textnormal{if $s_0$ has order } 2\\
	\{(1,1),(2,2),(3,3)\}& \textnormal{if $s_0$ has order } 3\\
	\{(1,0),(2,0),(2,2)\}& \textnormal{if $s_0$ has order } 4\\
	\{(1,0),(2,0),(2,1)\}& \textnormal{if $s_0$ has order } \geq 5.
\end{cases}\]
\end{enumerate}

Moreover, the value of $\bigl(\Triangles{s_0}{S'}-\Triangles{s_0}{S},\Triangles{s_0^2}{S'}-\Triangles{s_0^2}{S}\bigr)$ is the same for all finite generating sets $S$ containing $s_0$ and $s_0^2$ and satisfying~\eqref{eq:assumption_on_order}.

Finally, if $s_0$ is not of order $2$ and $G$ has an element of order $o$, with
\begin{equation}\label{eq:assumption_on_order2}\begin{cases}
    o \geq 2(2\abs S^2+3\abs S-2)^2(2\abs S^2+2\abs S) & \textnormal{if $o$ is odd}\\
    o \geq \Orietendorderfct(\abs S)=4(2\abs S^2+3\abs S-2)^2\abs S(\abs S+2)& \textnormal{otherwise}.
\end{cases}\end{equation}
then it is even possible to find $S'$ as above and such that $S'\setminus S$ contains no involution.
\end{lemma}
\begin{proof}
Let $\gamma$ be the element of large order given in the hypothesis and $\Delta_n\defegal\{\gamma^n,\gamma^{-n},s_0^{-1}\gamma^n,\gamma^{-n}s_0\}$.
We will show that there exists an integer $n$ such that $S'=S'_n\defegal S\cup\Delta_n$ works.
Observe that for all $n$, the set $S'_n$ satisfies Condition \ref{ConditionA} of the lemma.

We first restrict our attention to indices $n$ such that the following three conditions hold
\begin{gather}
	\abs{\gamma^n}_S\geq 3\label{Condition1}\\
	\abs{s_0^{-1}\gamma^n}_S\geq 3\label{Condition2}\\
	\gamma^{2n}\notin S\cup s_0S\label{Condition3}
\end{gather}
where $\abs g_S$ is the word length of $g$ relative to the generating set $S$.

We claim that the number of $1 \leq n <\ord(\gamma)$ such that one of the conditions \eqref{Condition1}--\eqref{Condition3} does not hold is at most $2(\abs{S}^2+\abs S-1)$.
Indeed, Condition \eqref{Condition1} means that $\gamma^n$ (which is different from $1$) avoids elements of length $1$ or $2$ and there are at most $\abs S^2$ such elements.
If Condition \eqref{Condition1} holds, it implies that $s_0^{-1}\gamma^n$ is of $S$-length at least $2$, therefore it is enough to avoid $(\abs S-1)^2$ new elements (the number of reduced $S$-words $s_1s_2$ of length $2$ such that $s_0s_1s_2$ has length $3$) to ensure that Condition \eqref{Condition2} also holds.
Finally, there are at most $2(2\abs S-1)$ values of $n<\ord(\gamma)$ such that Condition \eqref{Condition3} fails.
This almost gives the claim, as $\abs S^2 + (\abs S-1)^2+2(2\abs S-1)=2\abs{S}^2+2\abs S-1$.
One gains $1$ by noticing that the (possible) $n$ such that $\gamma^n = s_0$ was counted twice: once to ensure Condition \eqref{Condition1} and once to ensure Condition \eqref{Condition3}.

Observe that if $\ord(\gamma)$ is odd, or if $n\leq\frac12\ord(\gamma)$, there is only at most $2\abs S-1$ values of $n<\ord(\gamma)$ such that Condition \eqref{Condition3} fails.
In consequence, if $\ord(\gamma)$ is odd, or if $n\leq\frac12\ord(\gamma)$, the number of $1 \leq n <\ord(\gamma)$ such that one of the conditions \eqref{Condition1}--\eqref{Condition3} does not hold is at most $2\abs{S}^2-1$.

Also, for an $n$ satisfying Conditions \eqref{Condition1} to \eqref{Condition3}, $S'_n$ automatically satisfies Condition \ref{ConditionB}.
Moreover, in the Cayley graph of $G$ relative to $S'_n$, a triangle with a side labelled by  $s\in\Delta_n$ has at least another side  labelled by an element of $\Delta_n$, otherwise $s$ would have $S$-length at most $2$.
This implies that any edge labelled by $s\in\Delta_n$ belongs to at most $6$ triangles in $\unCayley{G}{S'}$, which is Condition \ref{ConditionC}.
Indeed, if one edge $e$ is labelled by $s\in\Delta_n$, there are at most three possibilities to put an edge labelled by $t\in \Delta_n\setminus\{s^{-1}\}$ at each extremity of $e$, thus giving a maximum number of $2\cdot 3=6$ triangles containing $e$.
This also shows that for any $s\in S$ we have 
\[
	\Triangles{s}{S'_n}-\Triangles{s}{S}=\abs{\setst{t\in\Delta_n}{s^{-1}t \in \Delta_n}}=\abs{\Delta_n\cap s\Delta_n}
\]

We now turn our attention on the set $\Delta_n\cap s\Delta_n$.
Its cardinality is equal to the number of pairs $(u,v) \in \Delta_n$ such that $u=sv$.
By replacing $u$ and $v$ by the words $\gamma^n,\gamma^{-n},s_0^{-1}\gamma^n$ and $\gamma^{-n}s_0$, this gives us $16$ equations in the group.
Among these $16$ equations, $4$ imply that $s=1$ and $4$ that $\gamma^{2n}$ belongs to $S\cup s_0S$, which is impossible if $n$ satisfies Conditions \eqref{Condition1} to \eqref{Condition3}.
The $8$ remaining equations for elements of $\Delta_n\cap s\Delta_n$ are shown in Table \ref{TableDeltan}, where $\alpha_n(g)\defegal \gamma^{-n}g\gamma^n$ and $\beta_n(g)\defegal \gamma^{-n}g\gamma^{-n}$.
\begin{table}
\[\arraycolsep=1.4pt\def\arraystretch{1.2}\begin{array}{|l|l|}
	\hline
	\textnormal{Possible elements of }\Delta_n\cap s\Delta_n & \textnormal{Occurs if}\\
	\hline\hline
	\gamma^n=ss_0^{-1}\gamma^n & s=s_0\\
	\hline
	s_0^{-1}\gamma^n=s\gamma^n & s=s_0^{-1}\\
	\hline
	\gamma^{-n}s_0=s\gamma^{-n} & s=\alpha_n(s_0)\\
	\hline
	\gamma^{-n}=s\gamma^{-n}s_0 & s=\alpha_n(s_0)^{-1} \\
	\hline
	\gamma^{-n}s_0=s\gamma^n & s=\beta_n(s_0)\\
	\hline
	\gamma^n=s\gamma^{-n}s_0 & s=\beta_n(s_0)^{-1}\\
	\hline
	\gamma^{-n}s_0=ss_0^{-1}\gamma^n & s=\beta_n(s_0)s_0\\
	\hline
	s_0^{-1}\gamma^n=s\gamma^{-n}s_0 & s=(\beta_n(s_0)s_0)^{-1}\\
	\hline
\end{array}\]
\caption{Possible elements of $\Delta_n\cap s\Delta_n$, where $\alpha_n(g)\defegal \gamma^{-n}g\gamma^n$ and $\beta_n(g)\defegal \gamma^{-n}g\gamma^{-n}$.}
\label{TableDeltan}
\end{table}
Observe that $\alpha$ and $\beta$ give two actions of $\Z/\ord(\gamma)\Z$ on $G$ (viewed as a set) and that $\gamma^{-n}s_0$ (and $s_0^{-1}\gamma^n$) is an involution if and only if $\beta_n(s_0)=s_0^{-1}$.
Let $A$, respectively $B$, denote the size of the orbit of $s_0$ under $\alpha$, respectively under $\beta$.
Let $M\defegal 2\abs S^2+3\abs S-2$.
By hypothesis we have $2M^2< \ord(\gamma)$.
On the other hand, $\gamma^{-(A\cdot B)}s_0\gamma^{A\cdot B}=s_0=\gamma^{-(A\cdot B)}s_0\gamma^{-(A\cdot B)}$ which implies that $\gamma^{2AB}=1$ and $2A\cdot B\geq\ord(\gamma)$.
This implies in particular that at least one of $A$ or $B$ is (strictly) greater than $M$, leaving us with $3$ cases.

\paragraph{Case 1}
Both $A$ and $B$ are greater than $M$.
If $n\leq M$ is such that
\begin{gather}\label{PropertyCase1}
\alpha_n(s_0)\notin S\quad \textnormal{and}\quad
\beta_n(s_0)\notin S\quad \textnormal{and}\quad
\beta_n(s_0)s_0\notin S
\end{gather}
then $\Delta_n\cap s\Delta_n$ contains at most $2$ elements: $\gamma^n$ if $s=s_0$ and $s_0^{-1}\gamma^n$ if $s=s_0^{-1}$.
This implies Condition \ref{ConditionD} and that the pair $\bigl(\Triangles{s_0}{S'}-\Triangles{s_0}{S},\Triangles{s_0^2}{S'}-\Triangles{s_0^2}{S}\bigr)$ is equal to $(2,0)$ if $s_0$ has order $2$, $(1,1)$ if $s_0$ has order $3$ and $(1,0)$ otherwise.
Moreover, neither $\gamma^n$ nor $\gamma^{-n}s_0$ are involutions since $2n\leq 2M< \ord(\gamma)$ and $\beta_n(s_0)$ is not in $S$.
We now prove that an $n$ satisfying \eqref{PropertyCase1} as well as Conditions \eqref{Condition1} to \eqref{Condition3} always exists if $M$ and the order of $\gamma$ are large enough.
Conditions \eqref{Condition1} to \eqref{Condition3} and \eqref{PropertyCase1} forbid some values of $n$ with $1\leq n\leq M$.
As already explained, since $n\leq M\leq\frac12\ord(\gamma)$, Conditions \eqref{Condition1} to \eqref{Condition3} forbid at most $2\abs S^2-1$ values of $n$.
On the other hand, all the $(\alpha_n(s_0))_{n=1}^M$ are distinct and distinct from $s_0$, so the condition $\alpha_n(s_0) \notin S$ forbids at most $\abs S-1$ values.
The same is true for $(\beta_n(s_0))_{n=1}^M$, and the condition for $(\beta_n(s_0)s_0)_{n=1}^M$ forbids at most $\abs{S}$ values.
Therefore, if both $A$ and $B$ are greater than $M$, then the number of $n\leq M$ such that $S'_n$ does not work is at most $(2\abs S^2-1)+(\abs S-1)+(\abs S-1)+\abs S=2\abs S^2+3\abs S-3$ which is strictly less than $M$ and this finishes the proof of Case 1.

\paragraph{Case 2}
If $A\leq M$ and $B> M$.
This time we shall take $n\leq B$ such that $n$ is a multiple of $A$, it satisfies Conditions \eqref{Condition1} to \eqref{Condition3} and both $\beta_n(s_0)$ and $\beta_n(s_0)s_0$ do not belong to $S$, in particular $\gamma^{-n}s_0$ is not an involution.
For such an $n$, the set $\Delta_n\cap s\Delta_n$ contains at most $4$ elements: $\gamma^n$ and $s\gamma^{-n}$ if $s=s_0$ and $\gamma^{-n}$ and $s \gamma^n$ if $s=s_0^{-1}$.
This implies Condition \ref{ConditionD} and that the pair $\bigl(\Triangles{s_0}{S'}-\Triangles{s_0}{S},\Triangles{s_0^2}{S'}-\Triangles{s_0^2}{S}\bigr)$ is equal to $(4,0)$ if $s_0$ is of order $2$, $(2,2)$ if $s_0$ is of order $3$ and $(2,0)$ otherwise.
So we are left to justify that such an $n$ exists.
As noted at the start of the proof, Conditions \eqref{Condition1} to \eqref{Condition3} forbid at most $2(\abs{S}^2+\abs S-1)$  values ($2\abs S^2-1$ if $\ord(\gamma)$ is odd) of $1 \leq n <B$.
Similarly, the conditions $\beta_n(s_0) \notin S$ and $\beta_n(s_0)s_0 \notin S$ forbid respectively at most $\abs{S}-1$ and $\abs{S}$ values, so we are done if
\[
	\left\lfloor\frac{B-1}{A}\right\rfloor > 2 \abs S^2 + 4\abs S -3,
\]
for example if $\frac{B}{A} \geq 2\abs{S}^2 +4\abs S-1$.
If we want to ensure that $\gamma^n$ is not an involution, we may need to forbid one more value of $n$ and take $\frac{B}{A} \geq 2\abs{S}^2 +4\abs S$.
Since $2AB\geq \ord(\gamma)$, we have
\[
	\frac{B}{A}\geq \frac{\ord(\gamma)}{2A^2} \geq\frac{\ord(\gamma)}{2M^2}
\]
Therefore, there exists an $n$ such that the conclusion of the lemma holds if $\ord(\gamma)\geq 2M^2 (2 \abs S^2 + 4\abs S -1)$.
If moreover we want to ensure that $\Delta$ contains no involution, we need to take $\ord(\gamma)\geq 4M^2 \abs S(\abs S + 2)$.
On the other hand, if we know that $\ord(\gamma)$ is odd, then it is enough to have $\ord(\gamma)\geq 4M^2 \abs S(\abs S+1)$.

\paragraph{Case 3}
If $A\geq M$ and $B< M$.
Similarly to Case 2, we take $n< A$ a multiple of $B$ satisfying Conditions \eqref{Condition1} to \eqref{Condition3} and that $\alpha_n(s_0)\notin S$.
Since $\beta_n(s_0)=s_0$, $\gamma^{-n}s_0$ is an involution if and only if $s_0$ is an involution.
Such an $n$ exists as soon as $\ord(\gamma)\geq2M^2(2\abs{S}^2+3\abs S-1)$.
On the other hand, if $\ord(\gamma)\geq2M^2\abs S(2\abs{S}+3)$ and $s_0$ is not of order $2$ it is possible to ensure that $\Delta$ contains no involution while it is enough to have $\ord(\gamma)\geq 4M^2 \abs S(\abs S+\frac 1 2)$ if $\ord(\gamma)$ is odd.
For such an $n$, the set $\Delta_n\cap s\Delta_n$ contains at most $4$ elements: $\gamma^n$ and $s\gamma^{n}$ if $s\in\{s_0,s_0^{-1}\}$, $\gamma^{-n}s_0$ if $s=s_0^2$ and $s_0^{-1}\gamma^n$ if $s=s_0^{-2}$.
This implies Condition \ref{ConditionD}  and that $\bigl(\Triangles{s_0}{S'}-\Triangles{s_0}{S},\Triangles{s_0^2}{S'}-\Triangles{s_0^2}{S}\bigr)$ is equal to $(2,0)$ if $s_0$ is of order $2$, $(3,3)$ if $s_0$ is of order $3$, $(2,2)$ if $s_0$ is of order $4$ and $(2,1)$ otherwise.
The only case that need explication is when $s_0$ is of order $4$.
In this case, $\bigl(\Triangles{s_0}{S'}-\Triangles{s_0}{S},\Triangles{s_0^2}{S'}-\Triangles{s_0^2}{S}\bigr)$ is equal to $(2,1)$ if  $s_0^{-1}\gamma^n=\gamma^{-n}s_0$ and $(2,2)$ otherwise.
But since we took $n$ as a multiple of $B$, we have $s_0=\beta_n(s_0)=\gamma^{-n}s_0\gamma^{-n}$ which forbids the case $(2,1)$.

\paragraph{}
In all of the $3$ cases, we can always find an $n$ such that the conclusion of the lemma holds as soon as $\ord(\gamma)\geq 2M^2 (2 \abs S^2 + 4\abs S -1)=2(2\abs S^2+3\abs S-2)^2(2 \abs S^2 + 4\abs S -1)$.
On the other hand, the bound $\ord(\gamma)\geq4\abs S(\abs S + 2)(2\abs S^2+3\abs S-2)^2$ ensures that there is no involution in $\Delta$, while the bound $\ord(\gamma)\geq 4(2\abs S^2+3\abs S-2)^2(\abs S^2+\abs S)$ is sufficient if $\ord(\gamma)$ is odd.

Moreover, by construction the value of $\bigl(\Triangles{s_0}{S'}-\Triangles{s_0}{S},\Triangles{s_0^2}{S'}-\Triangles{s_0^2}{S}\bigr)$ only depends on $A$ and $B$, which in turn depend on $s_0$ and $\gamma$ only.
 \end{proof}
\begin{lemma}\label{Lemma:9.3}
Recall that $\Evenorderfct(n) = 2(2n^2+3n-2)^2(2n^2+4n-1)$ and $\Orietendorderfct(n)=4(2n^2+3n-2)^2(n+2)n$.

Let $S\subseteq G\setminus\{1\}$ be a finite generating set.
Suppose that $G$ has an element of order at least $\Evenorderfct(15p+28q+2p^2+4pq+2q^2-4)$ where $p$ is the number of elements of order $2$ of $S$ and $q$ is the number of elements of order at least $3$ of $S$.
Then there exists a finite symmetric generating set $S\subseteq \tilde S\subseteq G\setminus\{1\}$ of size bounded by $15p+28q+2p^2+4pq+2q^2$ such that for all $s\in S$ and $t\in\tilde S$, if $\Triangles{s}{\tilde S}=\Triangles{t}{\tilde S}$ then $t=s$ or $t=s^{-1}$.

If $S$ has no elements of order $2$ and $G$ has an element of order a least $\Orietendorderfct(2\abs{S}^2+28\abs{S}-4)=\Orietendorderfct(\frac12\abs{S^\pm}^2+14\abs{S^\pm}-4)$, then there exists a finite symmetric generating set $S\subseteq \tilde S\subseteq G\setminus\{1\}$ with no elements of order $2$, of size bounded by $\frac12\abs S^2+14\abs S$ such that for all $s\in S$ and $t\in\tilde S$, if $\Triangles{s}{\tilde S}=\Triangles{t}{\tilde S}$ then $t=s$ or $t=s^{-1}$.
\end{lemma}
Before proving the lemma, we draw the attention of the reader to the fact that $15p+28q+2p^2+4pq+2q^2$ is smaller than both $2\abs{S^\pm}^2+15\abs{S^\pm}$ and $2\abs S^2+28\abs S$.
Depending on what is known on $S$, one bound may be better than the other.
In particular, Lemma \ref{Lemma:9.3} directly implies Theorem~\ref{Thm:Postrigidity}
\begin{proof}
To prove the first assertion, it is enough that all elements of $S$ belong to at least $7$ $\tilde S$-triangles (to distinguish them from the newly added elements which will belong to at most $6$ $\tilde S$-triangles) and that the numbers $\Triangles{s^{\pm1}}{\tilde S}$ for $s$ in $S$ are all distinct.
In order to achieve this, we will apply the first assertion of Lemma \ref{Lemma:9.2} several times to elements of $S$ in order to augment the number of triangles to which they belong.
This will give us a sequence of generating sets $S^\pm=S_0\subseteq S_1\subseteq\dots\subseteq S_k=\tilde S$ where $S_{i+1}=S_i'$ from Lemma \ref{Lemma:9.2} and $k$ is the total number of times we apply Lemma \ref{Lemma:9.2}.
Since the cardinality of the $\abs{S_i}$ is an increasing sequence, for $0\leq i<k$, it is sufficient to have an element of order at least $\Evenorderfct(\abs{S_{k-1}})$ to apply Lemma \ref{Lemma:9.2} to $S_i$.
On the other hand, since each application of Lemma \ref{Lemma:9.2} adds at most $4$ elements, we also have $\abs{S_{k-1}}\leq \abs{S^\pm}+4(k-1)$.
Altogether, this means that the existence of an element of order $\Evenorderfct\bigl(\abs{S^\pm}+4(k-1)\bigr)$ implies the conclusion of Lemma \ref{Lemma:9.3}.
It is thus enough to determine $k$ to finish the proof.
Finally, we use the fact that $p+q\leq\abs{S}$ while $p+2q=\abs{S^\pm}$.

The proof of the second assertion of Lemma \ref{Lemma:9.3} is similar except for the fact that $\abs{S^\pm}=2q$ while $q\leq\abs{S}$ and that we need to use the last assertion of Lemma \ref{Lemma:9.2} (that is the function $\Orietendorderfct$ instead of the function $\Evenorderfct$) to ensure that we do not add elements of order $2$ to~$S$.

The first assertion of Lemma \ref{Lemma:9.2} tells us that we can augment the number of triangles to which $s\in S^\pm$ and $s^{-1}$ belong, without changing the number of triangles for other $t\in S^\pm$, except maybe for $t\in\{s^2,s^{-2}\}$.
Moreover, in doing that, the new elements we add to $S^\pm$ belong to at most $6$ triangles at the moment they are added, and they cannot belong to more than $6$ later in process as they are never of the form $s^{\pm 2}$ for $s\in S$.

To be more precise, we define a directed graph $(V,E)$ as follows.
The vertices are the equivalence classes of elements of $S^\pm$ modulo the equivalence $s \sim t$ if $t=s$ or $t=s^{-1}$, so $\abs V=p+q$.
There is an arc $([s] \to [t])$ from the class of $s$ to the class of $t$ if $[t]\neq [s]$ and if, when applied to $s_0=s$, the generating set $S'$ given by Lemma \ref{Lemma:9.2} satisfies $\Triangles{t}{S'}-\Triangles{t}{S^\pm}>0$.
Note that this implies that $[t]=[s^2]$.
Moreover, by the second assertion of Lemma \ref{Lemma:9.2}, each time we use (the first assertion of) Lemma \ref{Lemma:9.2} with $s$, to go from $S_i$ to $S_{i+1}$, for $t \in S_i$ we have $\Triangles{t}{S_{i+1}}-\Triangles{t}{S_i} = \Triangles{t}{S'}-\Triangles{t}{S^\pm}$, which is positive if and only if $[t]=[s]$ or $([s] \to [t]) \in E$.
Observe that since $\Triangles{s}{S}=\Triangles{s^{-1}}{S}$, the number $\Triangles{[s]}{S}$ is well-defined.

An important observation at this point is that (as in every directed graph with out-degree bounded by $1$) the vertex set $V$ can be partitioned as $V = \{[s_1],\dots,[s_r]\} \sqcup C_1 \sqcup \dots \sqcup C_m$ where the $C_i$ are cycles and $F = \{[s_1],\dots,[s_r]\}$ is a forest: there is no arc from $[s_i]$ to $[s_j]$ for $j<i$.
In particular if we apply Lemma \ref{Lemma:9.2} to $s_i$, we will have $\Triangles{s_j}{S'}=\Triangles{s_j}{S^\pm}$ for $j<i$.

\paragraph{Initialization}
We first need to ensure that for every $s\in S^\pm$ we have $\Triangles{s}{\tilde S}\geq7$.
This can be done by applying Lemma \ref{Lemma:9.2} at most $4p+7q$ times.
Indeed for every $s_0 \in S^\pm$, each application of Lemma \ref{Lemma:9.2} augments both $\Triangles{s}{S}$ and $\Triangles{s^{-1}}{S}$ by $1$ if $s$ is of order at least $3$ and augments $\Triangles{s}{S}$ by $2$ if $s$ is of order $2$.

\paragraph{The forest}
We then deal with the elements $s_1,\dots,s_r \in S$.
Assume that, for some $1 \leq j <r$, we have constructed a finite generating set $\widetilde S_j$ containing $S$ with the following two properties:
\begin{gather}\label{PropertyForest}
\begin{cases}
\textnormal{every $s \in S$ belongs to at least $7$ triangles and}\\
\Triangles{s_i}{\widetilde S_j} \neq \Triangles{s_{i'}}{\widetilde S_j} \textnormal{ for every }i\neq i' \leq j
\end{cases}
\end{gather}
If $\Triangles{s_{j+1}}{\widetilde S_j} \notin \{\Triangles{s_i}{\widetilde S_j}: i \leq j\}$ we can put $\widetilde S_{j+1} =\widetilde S_j$ and we are done for $j+1$.
Otherwise, we apply Lemma \ref{Lemma:9.2} several times with $s_{j+1}$, until the number of triangles for $s_{j+1}$ is different than for all $s_i$ with $i\leq j$.
The number of applications of Lemma \ref{Lemma:9.2} is necessarily bounded by $j$, as each application increases the number of triangles for $s_{j+1}$, but not for $s_i$ with $i\leq j$.
On the other hand, for $j=1$ we have just proved the existence of such $\widetilde S_1$ obtained from $S$ after at most $4p+7q$ applications of Lemma \ref{Lemma:9.2}.
So at the end, we obtain a generating set $\widetilde S_r$ satisfying \eqref{PropertyForest} after applying Lemma \ref{Lemma:9.2} in total less than $ k_f:= 4p+7q+ \sum_{j=0}^{r-1} j$ times.

\paragraph{The cycles}
Finally, we have to treat the cycles separately.
Assume that, for some $0 \leq j < m$, we have obtained a generating set $\widetilde S_{r+j}$ containing $\widetilde S_r$ such that the function $\Triangles{\cdot}{\widetilde S_{r+j}}\colon V\to\mathbf N$ is injective on $F \sqcup C_1\sqcup \dots \sqcup C_j$.

Denote by $M_j$ the cardinality of $F \sqcup C_1\sqcup \dots \sqcup C_j$, and let $c_{j+1}\defegal\abs{C_{j+1}}=M_{j+1}-M_j$.
Without loss of generality, we may suppose that $c_1\leq\dots\leq c_m$.
Let $s \in S$ such that $[s] \in C_{j+1}$.
Its order is at least $5$, and $C_{j+1} = \{[s],[s^2],\dots,[s^{2^{c_{j+1}-1}}]\}$, with $s^{2^{c_{j+1}}}=s^{\pm 1}$.
By inspecting the conclusion \ref{item:cases_Lemma92} of Lemma \ref{Lemma:9.2} and recalling the definition of the graph $(V,E)$, we observe  that each application of Lemma \ref{Lemma:9.2} for $s$ increases the number of triangles for $s$ by $2$ and for $s^2$ by $1$.

If $c_{j+1}=2$, we apply Lemma \ref{Lemma:9.2} for $s$ until the number of triangles to which $s$ belongs is different from the number for $s^2$, and both are different from the number for each $[t] \in F \sqcup C_1\sqcup \dots\sqcup C_j$.
This requires to apply Lemma \ref{Lemma:9.2} at most $1+2M_j$ times, because each application of Lemma \ref{Lemma:9.2} increases by $2$ the number of triangles for $s$, by $1$ the number of triangles for $s^2$ and leaves unchanged the number of triangles containing $t$ for each  $[t] \in F \sqcup C_1\sqcup \dots \sqcup C_j$.

If $c_{j+1} >2$, we first apply Lemma \ref{Lemma:9.2} for $s^{2^{c_{j+1}-2}}$ until the number of triangles to which $s^{2^{c_{j+1}-1}}$ belongs is different from the number for each $[t] \in F \sqcup C_1\sqcup \dots\sqcup C_j$ (and we do not care yet about the number of triangles containing $s^{2^{c_{j+1}-2}}$).
This requires at most $M_j$ applications.
If $c_{j+1} >3$, we then apply Lemma \ref{Lemma:9.2} at most $M_j+1$ times to $s^{2^{c_{j+1}-3}}$ to ensure that the number of triangles for $s^{2^{c_{j+1}-2}},s^{2^{c_{j+1}-1}}$ and $[t] \in F \sqcup C_1\sqcup \dots\sqcup C_j$ are all different.
We go on, until we apply Lemma \ref{Lemma:9.2} at most $M_j+c_{j+1}-3$ times to $s^2$ so that that the number of triangles for $s^{2^{2}},\dots,s^{2^{c_{j+1}-1}}$ and $[t] \in F \sqcup C_1\sqcup \dots\sqcup C_j$ are all different.
We then, as in the case $c_j=2$, apply Lemma \ref{Lemma:9.2} at most $1+2(M_j+c_{j+1}-2)=(M_j+c_{j+1}-2)+(M_j+c_{j+1}-1)$ times and ensure that the number of triangles for $s, s^2,\dots,s^{2^{c_{j+1}-1}}$ and $t \in F \sqcup C_1\sqcup \dots\sqcup C_j$ are all different.
To summarize: given $\widetilde S_{r+j}$ we can construct the desired $\widetilde S_{r+j+1}$ by applying Lemma \ref{Lemma:9.2} at most
\[
	M_j+(M_j+1)+\dots+(M_j+c_{j+1}-1) = \sum_{l=M_j}^{M_{j}+c_{j+1}-1} l= \sum_{l=M_j}^{M_{j+1}-1} l
\]

For $j=m$, the generating set $\widetilde S_{r+m}$ satisfies that the function $\Triangles{\cdot}{\widetilde S_{r+m}}$ is injective on $V$ and takes values in $\{7,8,\dots\}$, as requested.

To obtain this result, we used Lemma \ref{Lemma:9.2} $k$ times, where $k$ is bounded above by
\[ 4p+7q+\sum_{l=1}^{\abs V-1} l = 4p+7q+\frac12(p+q)(p+q-1)\]
where, we used the fact, already stated, that $\abs V=p+q$.
To do that, it is sufficient to have an element of order at least $\Evenorderfct\bigl(\abs{S^\pm}+4(k-1)\bigr)\leq\Evenorderfct(15p+28q+2p^2+4pq+2q^2-4)$.
 \end{proof}
%
%
%
%
%
%
%
%
%
%
\subsection{Tarski monsters}
The proof of Theorem \ref{Thm:Tarski2} and \ref{Thm:Tarski} will follow from the following variant of Theorem \ref{Thm:Postrigidity} with much better bounds and simpler proof, but for very specific groups. Observe that the proposition is stated for Tarski monsters, but it holds true with the same proof for a nonabelian group generated by two elements $a$ and $b$ such that $a$, $b$ and $ab$ are of order a prime number $p > 263$.

\begin{proposition}\label{proposition:StrongcolourRigid_Tarski} Assume that $p>263$, and let $\Tarski_p$ be a Tarski monster of exponent $p$. Let $a,b$ be two non-commuting elements of $\Tarski_p$, and define $c=ab$. Let $S=\{a,b,c\}$. There exist integers $2 \leq i, j,k \leq \frac{p}{2}$ such that $(G,S,T)$ is a strong colour-rigid triple, where 
  \[ T= S\cup \{a^i,a^{i+1},a^{i+2},a^{i+3},b^j,b^{j+1},b^{j+2},c^k,c^{k+1}\}.\]
\end{proposition}
\begin{proof}
  Define the following four generating sets
  \[ S_0=S, S_1 = S_0 \cup \{a^i,a^{i+1},a^{i+2},a^{i+3}\},\]
  \[S_2 = S_1 \cup \{b^j,b^{j+1},b^{j+2}\}, S_3 =S_2 \cup \{c^k,c^{k+1}\}.\]
Our aim is to prove the existence of $i,j,k$ such that
  \[ \Triangles{a}{S_3}= 7, \Triangles{b}{S_3}= 5, \Triangles{c}{S_3} = 3\]
and $\Triangles{s}{S_3} \leq 2$ for each $s \in S_3 \setminus S_0$. This will prove the Proposition because $S_3=T$.

First observe that $\Triangles{a}{S_0}=\Triangles{b}{S_0}=\Triangles{c}{S_0} = 1$. Indeed, $\Triangles{a}{S_0}$ is equal to the cardinality of $S_0^\pm \cap a^{-1} S_0^\pm$ where
\[a^{-1} S_0^{\pm} = \{1,a^{-2},a^{-1}b,a^{-1}b^{-1},b,a^{-1} c^{-1}\},\] so the observation boils down to the fact that $\{1,a^{-2},a^{-1}b,a^{-1}b^{-1},a^{-1}c^{-1}\} \cap S_0^{\pm}$ is empty. It is clear that $1$ and $a^{-2}$ do not belong to $S_0^{\pm}$. Let us for example justify why $a^{-1} c^{-1}$ does not belong to $S_0^{\pm}$. If $a^{-1} c^{-1}$ did belong to $\{a^{\pm 1},c^{\pm 1}\}$, this would imply that $a$ and $c$ commute, and (since $a$ and $c$ generate $\Tarski_p$) that $\Tarski_p$ would be abelian, which is not the case. Similarly, the equality $a^{-1} c^{-1} = b^{-1}$ is equivalent to $b a b^{-1} = a^{-1}$, which implies that $\Tarski_p$ is finite. Finally $a^{-1} c^{-1} = b$ is equivalent to $c^2=1$, in contradiction with the assumption that $c$ is of order $p$. This shows that $a^{-1} c^{-1}$ does not belong to $S_0^{\leq 2}$. Similar arguments show that $a^{-1}b$ and $a^{-1}b^{-1}$ do not belong to $S_0^{\leq 2}$ either, and that $\Triangles{a}{S_0}=1$. Replacing the triple $\{a,b,c=ab\}$ by $\{b,c^{-1} , a^{-1} = b c^{-1}\}$ (respectively $\{c,b^{-1},a=cb^{-1}\}$) we obtain that $\Triangles{b}{S_0}=\Triangles{c}{S_0} = 1$.

It is also clear that, if $2 \leq i <\frac{p-5}{2}$, $\Triangles{a}{T}\geq \Triangles{a}{S_1} \geq 7$, as adding $a^{i},a^{i+1},a^{i+2}$ and $a^{i+3}$ to the generating set we have added the $6$ distinct triangles $\{1,a,a^{i+1}\}$, $\{1,a,a^{i+2}\}$, $\{1,a,a^{i+3}\}$, $\{1,a,a^{-i}\}$, $\{1,a,a^{-i-1}\}$ and $\{1,a,a^{-i-2}\}$ which contain the edge $\{(1,a)\}$. Similarly, we have $\Triangles{b}{T} \geq 5$ and $\Triangles{c}{T} \geq 3$ if $2 \leq j <\frac{p-3}{2}$ and $2\leq k < \frac{p-1}{2}$. So we have to prove that, for suitable choices of $i,j,k$, we have not created any other triangle.

To do so, we will arrange that
\begin{enumerate}
\item\label{item:condition1} $\{a^i,a^{i+1},a^{i+2},a^{i+3}\} \cap S_0^{\leq 2} = \emptyset$, $\{b^j,b^{j+1},b^{j+2}\} \cap S_1^{\leq 2} = \emptyset$ and $\{c^k,c^{k+1} \} \cap S_2^{\leq 2} = \emptyset$ and
\item\label{item:condition2} $i<\frac{p-7}{2}$, $j<\frac{p-5}{2}$ and $k<\frac{p-3}{2}$ and
\item\label{item:condition3} $i>6$, $j>4$, $k>2$ and
\item\label{item:condition4} $i$,$i-1$,$i-2$,$j$,$j-1$, $k$ are all different from $\lfloor \frac{p}{3}\rfloor$.
\end{enumerate}

The first two conditions for $a$ ensure the following: passing from $S_0$ to $S_1$, we have that, $\Triangles{a}{S_{1}} = \Triangles{a}{S_{0}}+6$ and $\Triangles{s}{S_{1}} = \Triangles{s}{S_0}$ if $s=b$ or $c$. Indeed, the first condition means that the only new triangles created when passing from $S_0$ to $S_{1}$ are those which have at least two new edges. Since the new edges are labelled by powers of $a$, the third edge has to be labelled by a power $a$ as well, that is by $a$, $a^{-1}$ or by an element of $S_1 \setminus S_0$. The condition $i<\frac{p-7}{2}$ means that the number of new triangles containing the edge $(1,a)$ is exactly $6$ (for example it would be $7$ if $i=\frac{p-7}{2}$, as $a a^{i+3}=a^{-i-3}$). Similarly, the condition for $b$ (respectively $c$) ensures that $\Triangles{b}{S_{2}} = \Triangles{b}{S_{1}}+4$ and $\Triangles{s}{S_{2}} = \Triangles{s}{S_{1}}$ if $s \in S_1 \setminus\{b\}$ (respectively $\Triangles{c}{S_{3}} = \Triangles{c}{S_{2}}+2$ and $\Triangles{s}{S_{3}} = \Triangles{s}{S_{2}}$ if $s \in S_2 \setminus\{c\}$).

The last two conditions for $a$ ensure that $\Triangles{s}{S_{1}} \leq 2$ if $s \in S_{1} \setminus S_0$, and similarly for $S_2$ and $S_3$. Indeed, they imply that we have not created new triangles with all three edges new, as such a triangle would correspond to a triple $x,y,z \in \{a^i,a^{i+1},a^{i+2},a^{i+3}\}^\pm$ satisfying $xyz=1$; the third condition ensures that that there is no such $x,y,z$ with $x,y,z^{-1} \in \{a^i,a^{i+1},a^{i+2},a^{i+3}\}$, whereas the fourth condition ensures that there are no such $x,y,z$ all in $\{a^i,a^{i+1},a^{i+2},a^{i+3}\}$. 

So we are left to prove that we can indeed find $i,j,k$ satisfying conditions (\ref{item:condition1}-\ref{item:condition4}). Observe that the conditions for $i$ do not depend on $j$ and $k$, that the conditions for $j$ depend only on $i$, and that the conditions for $k$ depend on $i$ and $j$. So there are three claim: first we claim that there exists $i$ satisfying  (\ref{item:condition1}-\ref{item:condition4}). Then $i$ being fixed, we claim that there is $j$ satisfying (\ref{item:condition1}-\ref{item:condition4}). And then such a $j$ being fixed, we claim that there is $k$ satisfying (\ref{item:condition1}-\ref{item:condition4}). All three claims are proved similarly. Let us focus on the last one, which is where the condition $p>263$ appears.

So assume that we have found $i,j$ satisfying (\ref{item:condition1}-\ref{item:condition4}). We have to understand which powers of $c$ can appear in $S_2^{\leq 2}$. Clearly, apart from $1$ or $c^{\pm 2}$, there cannot be a power of $c$ of the form $xy$ for $x \in \{c^\pm\}$ and $y \in S_2$ or $x \in S_2$ and $y \in \{c^\pm\}$. So the powers of $c$ in $S_2^{\pm}$ are either $c^{\pm 1}, c^{\pm 2}$, or of the form $a^rb^s$ or $b^s a^r$ for $r \in \pm\{1,i,i+1,i+2,i+3\}$ and $s \in \pm \{1,j,j+1,j+2\}$. Actually, $a b^s=cb^{s-1}$ cannot be a power of $c$ distinct from $c$ itself (as $b$ and $c$ do not commute), and similarly for $a^rb$, $b^{-1}a^r$ and $b^{s}a^{-1}$. So we can bound the number of powers (distinct from $c^{\pm 1}$ and $c^{\pm 2}$) in $S_2^{\leq 2}$ by $2\times 9\times 7$. Since $S_2^{\pm 2}$ is symmetric, this implies that there are at most $9 \times 7=63$ elements of the form $c^t$ for $2 <t\leq \frac{p-1}{2}$ in $S_2^{\pm}$. So condition (\ref{item:condition1}) forbids at most $2\times 63=126$ values of $k$ among the $\frac{p-11}{2}$ values allowed by conditions (\ref{item:condition2}-\ref{item:condition4}). This leaves the place to at least one $k$ because $\frac{p-11}{2} > 126$.
\end{proof}
  
We now prove Theorem \ref{Thm:Tarski2} and Theorem \ref{Thm:Tarski}. Let $p>263$ be a prime number,
$\Tarski_p$ be a Tarski monster and $S=\{a,b\}$ be any generating set
of size~$2$.  Then both $a$ and $b$ have order $p$.  Moreover, the
normalizer of $a$ in $\Tarski_p$ is $\gen{a}$, the normalizer of $b$
is $\gen{b}$ and $\gen{b}\cap\gen{a}=\{1\}$. We have $a^2\neq b^2$ and
$aba^{-1}\notin\{b,b^{-1}\}$.  Therefore, we can take $g_a=b$ and
$g_b=a$ in the statement of Proposition \ref{Proposition:propertyR}
which gives us that for $S_0\defegal\{a,b,a^{-1}b\}$, the triple
$(\Tarski_p,S,S_0)$ is strongly orientation-rigid. If $T$ is the generating set of cardinality $12$ given by Proposition \ref{proposition:StrongcolourRigid_Tarski}, we have that $(G,S_0,T)$ is strongly colour-rigid. So by Lemma \ref{lemma:strongPtriplePpair}, $(G,S,T)$ is a strong \GRR{} triple. In particular $\unCayley{G}{R}$ is a \GRR{}, which proves Theorem \ref{Thm:Tarski2}. Proposition \ref{Prop:Coverings} also applies: for every covering $\psi\colon\unCayley{G}{T}\to \Delta$ that is bijective on balls of radius $1$, there exists a subgraph $\tilde\Delta$ of $\Delta$ such that the restriction of $\psi$ to $\unCayley{G}{S}$ is a covering onto $\tilde\Delta$ that is compatible with the labels of $\unCayley{G}{S}$. By \cite{MR3463202}, there are only three possibilities for such a covering.
Either it is the identity, or $\tilde\Delta$ is infinite with finite orbits under its automorphism group, or $\tilde\Delta$ consists of one vertex with loops (and is therefore not a simple graph). Since $\Aut(\Delta)$ injects in $\Aut(\tilde \Delta)$ by Proposition \ref{Prop:Coverings}, we obtain that either $\psi$ is the identity, or $\Delta$ is infinite with finite orbits under $\Aut(\Delta)$. This proves Theorem \ref{Thm:Tarski}.
%
%
%
%
%
%
%
\providecommand{\noopsort}[1]{} \def\cprime{$'$}


\begin{thebibliography}{10}

\bibitem{MR1149884}
S.~I. Adyan and I.~G. Lys\"enok.
\newblock Groups, all of whose proper subgroups are finite cyclic.
\newblock {\em Izv. Akad. Nauk SSSR Ser. Mat.}, 55(5):933--990, 1991.

\bibitem{MR0498225}
L.~Babai.
\newblock Infinite digraphs with given regular automorphism groups.
\newblock {\em J. Combin. Theory Ser. B}, 25(1):26--46, 1978.

\bibitem{MR603394}
L.~Babai.
\newblock Finite digraphs with given regular automorphism groups.
\newblock {\em Period. Math. Hungar.}, 11(4):257--270, 1980.

\bibitem{MR656006}
L\'{a}szl\'{o} Babai and Chris~D. Godsil.
\newblock On the automorphism groups of almost all {C}ayley graphs.
\newblock {\em European J. Combin.}, 3(1):9--15, 1982.

\bibitem{Benjamini}
Itai Benjamini.
\newblock On graph limits and random metric spaces.
\newblock Mini-Course at the workshop Graphs and Groups, Lille, France, 2012.

\bibitem{MR3027684}
David~P. Byrne, Matthew~J. Donner, and Thomas~Q. Sibley.
\newblock Groups of graphs of groups.
\newblock {\em Beitr. Algebra Geom.}, 54(1):323--332, 2013.

\bibitem{MR3604062}
Edward Dobson, Ademir Hujdurovi\'{c}, Klavdija Kutnar, and Joy Morris.
\newblock On color-preserving automorphisms of {C}ayley graphs of odd
  square-free order.
\newblock {\em J. Algebraic Combin.}, 45(2):407--422, 2017.

\bibitem{MR3537032}
Edward Dobson, Pablo Spiga, and Gabriel Verret.
\newblock Cayley graphs on abelian groups.
\newblock {\em Combinatorica}, 36(4):371--393, 2016.

\bibitem{MR3864735}
John~Kevin Doyle, Thomas~W. Tucker, and Mark~E. Watkins.
\newblock Graphical {F}robenius representations.
\newblock {\em J. Algebraic Combin.}, 48(3):405--428, 2018.

\bibitem{MR0295919}
P.~Erd\H{o}s and A.~R\'{e}nyi, editors.
\newblock {\em Combinatorial theory and its applications. {I}-{III}}.
\newblock North-Holland Publishing Co., Amsterdam-London, 1970.
\newblock Colloquia Mathematica Societatis J\'{a}nos Bolyai, 4.

\bibitem{MR642043}
C.~D. Godsil.
\newblock G{RR}s for nonsolvable groups.
\newblock In {\em Algebraic methods in graph theory, {V}ol. {I}, {II}
  ({S}zeged, 1978)}, volume~25 of {\em Colloq. Math. Soc. J\'{a}nos Bolyai},
  pages 221--239. North-Holland, Amsterdam-New York, 1981.

\bibitem{HetzelThese}
D.~Hetzel.
\newblock {\em \"Uber regul\"are graphische Darstellung von aufl\"osbaren
  Gruppen}.
\newblock PhD thesis, Technische Universit\"at Berlin, 1976.

\bibitem{MR0457275}
W.~Imrich and M.~E. Watkins.
\newblock On automorphism groups of {C}ayley graphs.
\newblock {\em Period. Math. Hungar.}, 7(3-4):243--258, 1976.

\bibitem{MR0255446}
Wilfried Imrich.
\newblock Graphen mit transitiver {A}utomorphismengruppe.
\newblock {\em Monatsh. Math.}, 73:341--347, 1969.

\bibitem{MR0392660}
Wilfried Imrich.
\newblock On graphs with regular groups.
\newblock {\em J. Combinatorial Theory Ser. B}, 19(2):174--180, 1975.

\bibitem{MR3463202}
Paul-Henry Leemann.
\newblock {S}chreir graphs: {T}ransitivity and coverings.
\newblock {\em Internat. J. Algebra Comput.}, 26(1):69--93, 2016.

\bibitem{MR3751067}
Luke Morgan, Joy Morris, and Gabriel Verret.
\newblock Characterising {CCA} {S}ylow cyclic groups whose order is not
  divisible by four.
\newblock {\em Ars Math. Contemp.}, 14(1):83--95, 2018.

\bibitem{MR3667669}
Joy Morris and Pablo Spiga.
\newblock Every finite non-solvable group admits an oriented regular
  representation.
\newblock {\em J. Combin. Theory Ser. B}, 126:198--234, 2017.

\bibitem{2018arXiv181107709M}
Joy {Morris} and Pablo {Spiga}.
\newblock {Asymptotic enumeration of Cayley digraphs}.
\newblock {\em ArXiv e-prints}, page arXiv:1811.07709, November 2018.

\bibitem{MR3873496}
Joy Morris and Pablo Spiga.
\newblock Classification of finite groups that admit an oriented regular
  representation.
\newblock {\em Bull. Lond. Math. Soc.}, 50(5):811--831, 2018.

\bibitem{MR3266284}
Joy Morris, Pablo Spiga, and Gabriel Verret.
\newblock Automorphisms of {C}ayley graphs on generalised dicyclic groups.
\newblock {\em European J. Combin.}, 43:68--81, 2015.

\bibitem{MR0319804}
Lewis~A. Nowitz and Mark~E. Watkins.
\newblock Graphical regular representations of non-abelian groups. {I}, {II}.
\newblock {\em Canad. J. Math.}, 24:993--1008; ibid. 24 (1972), 1009--1018,
  1972.

\bibitem{MR571100}
A.~Yu. Ol\cprime~shanskiĭ.
\newblock An infinite group with subgroups of prime orders.
\newblock {\em Izv. Akad. Nauk SSSR Ser. Mat.}, 44(2):309--321, 479, 1980.

\bibitem{delaSalleTessera}
Mikael {\noopsort{Salle}de la Salle} and Romain {Tessera}.
\newblock Characterizing a vertex-transitive graph by a large ball.
\newblock {\em ArXiv e-prints}, 08 2015.

\bibitem{MR0280416}
Mark~E. Watkins.
\newblock On the action of non-{A}belian groups on graphs.
\newblock {\em J. Combinatorial Theory Ser. B}, 11:95--104, 1971.

\bibitem{MR0363998}
Mark~E. Watkins.
\newblock On graphical regular representations of {$C_{n}\times Q$}.
\newblock pages 305--311. Lecture Notes in Math., Vol. 303, 1972.

\bibitem{MR0344157}
Mark~E. Watkins.
\newblock Graphical regular representations of alternating, symmetric, and
  miscellaneous small groups.
\newblock {\em Aequationes Math.}, 11:40--50, 1974.

\bibitem{MR0422076}
Mark~E. Watkins.
\newblock Graphical regular representations of free products of groups.
\newblock {\em J. Combinatorial Theory Ser. B}, 21(1):47--56, 1976.

\bibitem{MR1159227}
Efim~I. Zelmanov.
\newblock On the restricted {B}urnside problem.
\newblock In {\em Proceedings of the {I}nternational {C}ongress of
  {M}athematicians, {V}ol. {I}, {II} ({K}yoto, 1990)}, pages 395--402. Math.
  Soc. Japan, Tokyo, 1991.

\bibitem{MR1119009}
Efim.~I. Zelmanov.
\newblock Solution of the restricted {B}urnside problem for {$2$}-groups.
\newblock {\em Mat. Sb.}, 182(4):568--592, 1991.

\end{thebibliography}
\end{document}